%% file: MBF_manuscript.tex
\documentclass[12pt]{amsart}
\usepackage[top=1in, bottom=1in, left=1in, right=1in]{geometry}
\usepackage{amsmath}
\usepackage{amsfonts}
\usepackage{amssymb}
\usepackage{amsthm}
\usepackage[usenames, dvipsnames]{xcolor}
\usepackage{tikz}
\usepackage{tikz-cd}

\pgfdeclarelayer{fg} 
\pgfsetlayers{main,fg} 

\usetikzlibrary{automata,topaths,arrows.meta,shapes,positioning,calc,decorations.pathreplacing,shapes.multipart}

\usepackage{graphicx}
\usepackage{caption}
\usepackage{pgfplots}
\usepackage{float}
\usepackage{mathrsfs}
\usepackage{wrapfig}
\usepackage{hyperref}
\usepackage[backend=bibtex]{biblatex}
\usepackage{enumitem}
\usepackage[normalem]{ulem}
\usepackage{booktabs}
\bibliography{bigdata1}

\usepackage{lineno}

\include{KM_definitions_7_2018}

\newcommand{\ep}{\epsilon} 

\newcommand{\bbR}{\mathbb{R}} 

\newcommand{\abs}[1]{\left| #1 \right|} 
\usepackage{color, colortbl}
\definecolor{forestgreen}{rgb}{0.13, 0.55, 0.13}

\definecolor{darkgray}{rgb}{.55,.55,.55}
\definecolor{lightgray}{rgb}{.8,.8,.8}
\definecolor{verylightgray}{rgb}{1,1,1}

\newcommand{\ttCol}{\mathtt{Col}}

\newcommand{\Truth}{\text{\sf True}}
\newcommand{\False}{\text{\sf False}}

\definecolor{applegreen}{rgb}{0.55, 0.71, 0.0}

\newcommand{\MBF}{\mathbf{MBF}}

\title{Joint Realizability of Monotone Boolean Functions}
\author{Peter Crawford-Kahrl, Bree Cummins, and Tomas Gedeon}

\begin{document}

\begin{abstract}
The study of monotone Boolean functions (MBFs)  has a long history. We explore a connection between MBFs and ordinary differential equation (ODE) models of gene regulation, and, in particular, a problem of the \textit{realization} of an MBF as a function describing the state transition graph of an ODE.
We formulate a problem of \textit{ joint realizability}  of finite collections of MBFs by establishing a connection 
between the parameterized dynamics of a class of ODEs and a collection of MBFs. We pose a question of what collections of MBFs can be realized  by ODEs that belong  to nested classes defined by increased algebraic complexity of their right-hand sides. As we progressively restrict the algebraic form of the ODE, we show by a combination of theory and explicit examples that the class of jointly realizable functions strictly decreases.
Our results impact the study of regulatory network dynamics, as well as the classical area of MBFs. We conclude with a series of potential extensions and conjectures. 
\end{abstract}
\maketitle


\section{Introduction}
The study of Boolean functions in general and monotone Boolean functions in particular has a long history~\cite{church1940nunmerical,paull1960boolean,elgot1961truth,muroga1970enumeration,winder1962threshold,chow1961boolean,korshunov03}.
One area in which monotone Boolean functions (MBFs) have been used is in modeling the dynamics of gene regulatory networks. In these models the (Boolean) state of each node $i$ in the regulatory network is updated based on the (Boolean) states of the nodes $j$ that that are connected by an edge from $j$ to $i$. The monotonicity requirement on a Boolean function (Definition~\ref{pos}) reflects the fact that the edges in gene regulatory networks are signed and thus the effect of one gene on another is always either monotonically increasing (activating edge) or monotonically decreasing (repressing edge). 

An alternative class of network models uses continuous time dynamics of ordinary differential equations. We are interested in a particular class of such models  with piecewise linear right-hand sides~\cite{ThieffryThomas95,glass:kaufman:72, glass:kaufman:73,Thieffry06,edwards00,Bernot2007,Snoussi89}.
For the most general of these models, which we call \textit{$K$-systems}, 
the right-hand side is fully determined by a finite collection of constants $K = (K_1, \ldots, K_n)$, where $K_i$ is a collection of constants that describes the activity of node $i$ in the regulatory network.
Each collection $K_i$ within $K$ also satisfies a monotonicity condition  that reflects the monotone effect of the edges in the network.

The main goal of this paper  is to show that there is  a close  relationship between $K$-systems and collections of monotone Boolean functions. In order to show this connection, we first show that to each $K$-system one can associate 
a state transition graph (STG), which is a finite directed graph  that coarsely captures the progression of the trajectories of the $K$-system. There are a finite number of STGs, which permits an imposition of an  equivalence relation on the (infinite) set of $K$-systems, with an equivalence class denoted $[K]$.

Our first major result is the correspondence between the equivalence classes $[K]$ and collections of MBFs. 
For a fixed regulatory network with $n$ nodes, each equivalence class $[K]$ has the form $[K] = ([K_1], \ldots, [K_n])$.  Then each $[K_i]$ for a node $i$ with $m_i$ input edges and $b_i$ output edges corresponds to a collection of $b_i$ monotone Boolean functions with $m_i$ inputs. Moreover, each such collection of $b_i$ MBFs, satisfying an additional condition that the truth sets are linearly ordered by inclusion, is associated to an equivalence class  $[K_i]$. Using this result, the equivalence classes $[K]$ are arranged into a parameter graph (PG) specific to the  regulatory network under consideration. The edges of the PG are determined by the adjacency of the collections of MBFs associated to each $[K]$.  

Our next set of results examines the effect of imposing algebraic restrictions on the form of the right-hand side of the differential equations of the  network model, which results in additional structure on the set $K$.  These additional algebraic  restrictions decrease the size of the parameter graph.  We ask which MBFs, and which collections of MBFs, are realizable as parameter nodes of the corresponding restricted parameter graphs.

The classes of algebraic functions that we chose to  examine are nested; the most restricted and smallest class consists of linear functions, $\Sigma$, which is a subset of functions that can be obtained as products of sums of individual variables, $\Pi\Sigma$, and lastly sums of products of sums, $\Sigma\Pi\Sigma$. These classes are all special cases of $K$-systems and therefore admit STGs and PGs. These algebraic restrictions are motivated by the software DSGRN~\cite{Cummins16,Cummins2017b,Gedeon18,Gedeon2020}, which calculates the PGs and STGs for network models with the class of $\Pi\Sigma$ functions, and in principle can be extended to other classes of algebraic expressions.

We show that the classes $\Sigma$, $\Pi\Sigma$, and $\Sigma\Pi\Sigma$ do impose constraints on pairs of MBFs that can be realized as parameter nodes of a PG. In fact, we show that the classes of pairs of MBFs with three inputs that are realizable as linear functions are a strict subset of $\Pi \Sigma$-jointly realizable pairs, which is in turn a strict subset of  $\Sigma \Pi \Sigma$-jointly realizable pairs of MBFs. 
We also show that there are pairs of MBFs for any $n\geq 4$ inputs that are $K$-jointly realizable but are not $\Sigma \Pi \Sigma$ realizable. 

These results show that the increased complexity of the algebraic expression provides a richer class of models as measured by the set of MBFs that can be realized in a PG. At the same time, the connection between differential equation models and collections of MBFs allows for the formulation of a host of interesting questions (see the Discussion) about what $k$-tuples of MBFs can be realized as nodes of parameter graphs of differential equation models as the complexity of the right-hand side varies.

\section{$K$-systems and Monotone Boolean Functions}

A regulatory network is a  useful abstraction for organizing information about interacting units. Nodes represent units and (directed) edges interaction between the nodes. 

\begin{defn}\label{def:reg_net}
A \textit{regulatory network} {\bf RN} is a triple $\mathbf{ RN} := \{ V, E, s\}$ where
\begin{itemize}
\item $V$ is the set of \textit{vertices};
\item $E \subset V \times V$ is a finite set of oriented \textit{edges}, where  $(i,j)$ denotes the edge from $i$ to $j$;
\item $s: E \to  \{ +,-\}$ is the \textit{sign} of the edge.
\end{itemize}
We will generally use $n = |V|$.
We denote $S(i)$ to be the set of \textit{sources} of node $i$  and $T(i)$ the set of \textit{targets} of node $i$:
\[ 
S(i) = \{ j \in V \;|\; (j,i) \in E \}  \mbox{ and } T(i) = \{ j \in V \;|\; (i,j) \in E \}.
\] 
We split the set of sources into activating and repressing inputs as $S(i) = S(i)^+ \cup S(i)^-$ where 
\[j \in S^+(i) \mbox{ iff } e =(j,i)\in E \mbox{ and } s(e) = +\] and \[j \in S^-(i) \mbox{ iff } e =(j,i)\in E \mbox{ and } s(e) = -\]

\end{defn}
 The interpretation of the signed edges comes from biology; a positive edge signifies up-regulation, where the rate of change of the target node concentration increases as the concentration of the source node increases. A negative edge signifies down-regulation, where the rate of change of the target node concentration decreases as the concentration of  the source node increases. Inherent in this description is \textit{monotonicity} of the rate of change of the target node with respect to changes in each of the source nodes~\cite{albert:collins:glass,Gedeon2020,Heatha2009,Machado2011,Shamir2008,Kains2014a}.

One of the natural ways to associate dynamics to a network is using Boolean functions. Every node is assumed to be either OFF, corresponding to low concentration, represented by the  state $0$, or ON, corresponding to high concentration, represented by the state $1$.

\begin{defn}
Let $\B:= \{0,1\}$. We will use the notation $\B^n:= \{0,1\}^n$ for the vertices of a hypercube of dimension $n$. 
A \textit{Boolean function} is a function $f: \B^n \to \B$.
\end{defn}

In examples, we will often write elements of $\B^n$ as strings (e.g.\ $10010\in \B^5$).

\begin{defn}\cite{CH11} \label{pos}
A Boolean function $f : \B^n \to \B$  is \emph{positive}  (resp. \emph{negative}) in $x_i$ if $f|_{x_i =0} \leq f|_{x_i =1}$ (resp. $f|_{x_i =0} \geq f|_{x_i =1}$), where  $f|_{x_i =0}$ (resp.  $f|_{x_i =1}$)  denotes the value of $f(x_1,...,x_{i-1},0,x_{i+1},...,x_n)$  (resp. $ f(x_1,...,x_{i-1},1,x_{i+1},...,x_n)$) for any Boolean values of $x_1, \ldots, x_{i-1}, x_{i+1}, \ldots x_n$. We say that  $f$  is \textit{monotone  in $x_i $} if it is either positive or negative in $x_i$. $f $ is \textit{monotone} if it is monotone in $x_i$  for all $i \in \{1,...,n\}$.
\end{defn}

Positive and negative monotone Boolean functions (MBF) capture the effect of positive and negative edges in the network {\bf RN}, respectively.
We will use the notation
\begin{align*}
\MBF(n) & := \{f:\B^n\to\B \mid f \text{ is monotone}\} \\
\MBF^+(n) & := \{f:\B^n\to\B \mid f \text{ is positive in $x_i$ for all $i \in \{1,...,n\}$}\}
\end{align*}
The dynamics of the network with $n$ nodes  is described by iteration of  $f: \B^n\to \B^n$, where   $f:=(f_1, f_2, \ldots, f_n)$ is a collection of MBFs.
Monotone Boolean function models are widely used due to their simplicity, but matching their predictions to experimental values of continuous variables like concentration always poses a challenge. An effort to combine the simplicity of Boolean maps with a continuous time description was initiated  by~\cite{glass:kaufman:72,glass:kaufman:73,Thomas1973}. To explain this approach we extend the definition of regulatory network given in Definition~\ref{def:reg_net}.

  \begin{defn} \label{def:param_reg_net}
  A \textit{weighted regulatory network} is a regulatory network \textbf{RN} with positive, real-valued weights assigned to each node,
  \[ \gamma = (\gamma_1,\dots,\gamma_n) \]
  with $n = |V|$, and positive, real-valued weights assigned to each edge
  \begin{equation} \label{order} 
    \boldsymbol \theta_i = \{\theta_{ s_1 i}, \theta_{s_2 i}, \dots , \theta_{s_{b_i} i} \}  \mbox{ where } \{ s_1, s_2, \dots, s_{b_i} \} = T(i),
    \end{equation}
    with $\theta = \bigcup_i \theta_i$. We want to bring attention to the  the indexing we use: $\theta_{ji}$ is associated to the edge from $i$ to $j$, in the tradition of~\cite{Cummins16}. 
    The node weights are called \textit{decay rates} and the edge weights are called \textit{thresholds}.
    We assume that for each node $i$, the $b_i$  thresholds in the collection $\boldsymbol \theta_i$  are distinct. 
 \end{defn}

The idea of decay rates comes from biology and indicates how quickly a gene product will break down under natural cellular processes. A common model of enzymatic gene regulation 
is the sigmoidal Hill function model, which has a half-saturation value. These half-saturation values are sometimes treated as thresholds, here represented as weights on edges. An example weighted regulatory network is shown in Figure~\ref{fig:exampleRN}.

\begin{figure}
    \centering
   \begin{tikzpicture}[node distance = 3cm,every text node part/.style={align=left}]

\tikzset{my node/.style = {circle, draw, inner sep=1pt,minimum size=0pt}}

  \node[my node] (1) {$1$};
  \node[my node] (2) [right=of 1] {$2$};
  \node (3) [right=0.7cm of 2] {$\theta_{11}=3$\\$\theta_{12}=1.5$\\$\theta_{21}=2$};
  \node (4) [right=0.3cm of 3] {$\gamma_1 = 1$\\$\gamma_2=1$};

\tikzset{my edge/.style = {very thick,->,shorten >= 3pt,shorten <= 3pt}}

  \draw[my edge] [loop left, looseness=17] (1) to ["$\theta_{11}$"] (1);
  \draw[my edge] [bend right] (1) to ["$\theta_{21}$"'] (2);
  \draw[my edge] [bend right, -|] (2) to ["$\theta_{12}$"'] (1);
\end{tikzpicture}

\vspace{.5cm}
\begin{tabular}{ll}
$S^+(1) = \{1\} $ & $S^-(1)= \{2\}$  \\
$S^+(2) = \{1\} $ & $S^-(2)= \emptyset$  \\
\end{tabular}
    \caption{An example weighted regulatory network. Here $V =\{1,2\}$. We use $\to$ to denote a positive (activating) edge, and $\dashv$ to denote a negative (inhibiting) edge. We also illustrate the sets of sources for each node.}
    \label{fig:exampleRN}
\end{figure}
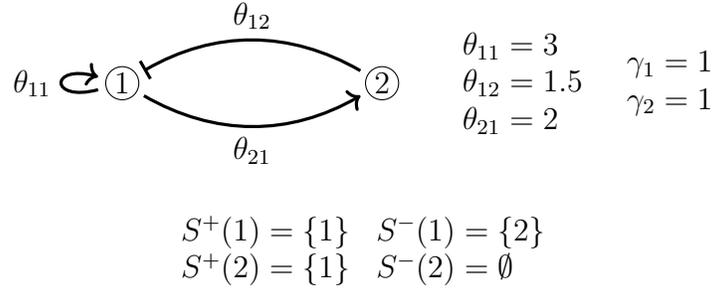

\subsection{K-systems} \label{subsection:Ksystem}

The most general  attempt to combine the  simplicity of Boolean maps with a continuous time description resulted in {\it switching $K$-systems}, consisting 
of a system of differential equations on $\bbR^n_+$. 
The ``$K$'' in $K$-system denotes  a finite collection of real values that satisfy a monotonicity assumption (Definition~\ref{def:Kmono}) and are used to parameterize a system of ordinary differential equations (ODEs) with  discontinuous right-hand sides. 

Given a weighted regulatory network, we associate to each node $i$ a continuous non-negative variable $x_i \in \bbR_+$, usually representing the concentration of a gene product. We write the collection  of gene concentrations as a vector $x = (x_1,\ldots,x_n) \in \bbR^n_+$. The thresholds $\boldsymbol \theta_{i}$ divide the $x_i$ axis into $b_i+1$ intervals, where  $b_i = |\boldsymbol \theta_i|$ is the number of targets of node $i$. We enumerate these by the integers $0, \ldots, b_i$ in ascending order. 
Then 
  \[X= \R^n_+ \setminus \{ x_i = \theta_{ji} \;|\; i \in {1, \dots, n}, \quad \theta_{ji} \in \boldsymbol \theta_i \}\] 
  is an open rectangular grid where each component of $X$ is an open domain. 
   As we will see next, the collection of  real numbers $K$ determines an ODE system defined on $X$ whose solutions are consistent with a discrete mapping between open domains in $X$ and this   discrete map can be interpreted as a collection of $\MBF$s.  The following definition of the $K$-system goes back to at least Thieffry and Romero~\cite{Thieffry99}; we follow the exposition in \cite{Bernot2007}. 

\begin{defn} \label{def:Kmono}
  Recall the definition of a regulatory network \textbf{RN} in Definition~\ref{def:reg_net}, particularly the nodes $V$ and the sources $S^\pm(i)$. Let
  \begin{equation}\label{eq:K}
  K := \{ K_{i,A,B} \in \bbR_+ \;|\;  i \in V, A \subset S^+(i), B \subset S^-(i) \}
  \end{equation}
  be a collection of non-negative numbers that satisfies the \emph{monotonicity assumption}:
  \begin{itemize}
  \item For each $i \in V$, if  $A \subsetneq A' \subset S^+(i)$ then 
  \[ K_{i,A,B} \leq K_{i,A',B} \quad \mbox{ for all } B \subset  S^-(i). \]
  \item For each $i \in V$, if  $B\subsetneq B' \subset S^-(i)$ then 
  \[ K_{i,A,B} \geq K_{i,A,B'} \quad \mbox{ for all } A \subset  S^+(i). \]
  \end{itemize}
 Fixing a collection of constants $K$ satisfying the monotonicity assumption, define a \emph{parameter assignment function} for each continuous variable $x_i$ from the power sets of $S^+(i)$ and $S^-(i)$ into the positive real numbers:
\[
k_i :  2^{S^+(i)} \times 2^{S^-(i)} \to \bbR_+ , \qquad k_i(A,B)  :=K_{i, A,B}
 \]
 We write  $k = (k_1,\dots,k_n)$ as the collection of these parameter assignment functions, one for each component of the system. 
    \end{defn}

  We continue the example from Figure~\ref{fig:exampleRN} by listing an example assignment of numbers $K$ that satisfy the monotonicity assumption. 
  
\begin{align}
& K_{1,\emptyset,2}=0.1  \quad K_{2,\emptyset,\emptyset}=0.2 \label{eq:Kexample}\\
& K_{1,\emptyset,\emptyset}=0.5\quad K_{2,1,\emptyset}=0.4 \nonumber\\
& K_{1,1,2}=5  \nonumber\\
& K_{1,1,\emptyset}=6  \nonumber
\end{align}
Up to now,the construction of $K$ has depended only on the structure of an unweighted regulatory network \textbf{RN}. We now take into account the weights associated to \textbf{RN} as in Definition~\ref{def:param_reg_net}. The K-system ODE, that we describe next, will depend on the these weights.

The dynamics of variables $x_i$ are  affected  by the incoming edges to node $i$ in the regulatory network \textbf{RN}. For each  $x_j \in S(i)$,  the value of $x_j$ is either above or below the threshold $\theta_{ij}$ assigned to the edge from $j$ to $i$ in the weighted regulatory network. If $x_j \in S^+(i)$ has an activating effect, then $x_j > \theta_{ij}$ implies that $x_i$ will be produced at a  greater rate than when $x_j < \theta_{ij}$. The inequalities are swapped  for a repressing effect, $x_j \in S^-(i)$.  
With this in mind, we define the \textit{activity function} for a node $i$ and $x \in X$, as follows: 
\begin{equation}\label{eq:zeta} \zeta_i: X \to 2^{S^+(i)} \times 2^{S^-(i)} , \qquad \zeta_i(x) = (A_i,B_i)\end{equation}
\begin{align*}
A_i &= \{ j \in S^+(i)  \:|\:  x_j > \theta_{ij} \} \\
B_i &= \{ j \in S^-(i)\:|\:  x_j > \theta_{ij} \} .
\end{align*}
The map $\zeta := (\zeta_1, \ldots, \zeta_n)$, defined on $X$, is constant on each open domain of $X$.

The composition of the parameter assignment function with the activity function, $k_i \circ \zeta_i$, assigns to a vector $x \in X$  a scalar parameter  $k_i(A,B) = k_i(\zeta_i(x))$ in the set $K$.  
 Recalling the decay rates $\gamma$ from Definition~\ref{def:param_reg_net}, we are now in a position to define a  differential equation parameterized by $K$ and defined on  $X$.

\begin{defn}\label{def:Ksystem}
The system 
\begin{equation} \label{Ksystem}
\dot{x}_i = - \gamma_i x_i + k_i(\zeta_i(x))
\end{equation}
is called the {\it $K$-system} on $X$ associated to the weighted regulatory network \textbf{RN}. 
\end{defn}
Note that since $\zeta$ is constant on the open domains of $X$, the differential equation is linear in each such domain. On the boundaries of the domains, the system is undefined due to  the discontinuity in $k_i(\zeta_i(x))$. However,  we extend the system by continuity, whenever possible, from $X$ to  $\R^n_+$.
The assumption of the non-negativity of $k$ guarantees  that the non-negative orthant $\bbR^n_+$ is positively invariant and the concentrations $x_i$  remain non-negative for all $t \geq 0$.

\subsubsection{State transition graph}

\begin{figure}
    \begin{minipage}{.49\textwidth}
        \centering
    \begin{tikzpicture}[scale=.8]
	\fill[fill=lightgray] (2,2) rectangle (4,4);	
	\draw[ultra thick,-] (0,4) node[anchor=east]{$x_2$} -- (0,0) -- (6,0) node[anchor=north] {$x_1$} ;
	\draw[dashed] (2,0) node [below]{$\theta_{21}$} -- ++(0,4);
	\draw[dashed] (4,0) node [below]{$\theta_{11}$} -- ++(0,4);
	\draw[dashed] (0,2) node[anchor=east]{$\theta_{12}$}-- ++(6,0);
	
\tikzset{point/.style = {circle, fill=black, inner sep=1.5pt, node contents={}}}
                
    \node (v1) at (3.4,2.7)     [point, label=above:$x$];
    \node (v2) at (.8,.4)     [point, label=above:$\frac{k\circ \zeta}{\gamma}(x) $];
\end{tikzpicture}
    \end{minipage}\hfill
    \begin{minipage}{.49\textwidth}
    \begin{tikzpicture}[ node distance = 1cm,every text node part/.style={align=left}]
\tikzset{my node/.style = {circle, draw, inner sep=1pt,minimum size=0pt}}

  \node[draw, rounded rectangle] (11) {$(1,1)$};
  \node[draw, rounded rectangle] (21) [right=of 11] {$(2,1)$};
  \node[draw, rounded rectangle] (31) [right=of 21] {$(3,1)$};
  \node[draw, rounded rectangle] (12) [above=of 11] {$(1,2)$};  
  \node[draw, rounded rectangle, fill=lightgray] (22) [right=of 12] {$(2,2)$};
  \node[draw, rounded rectangle] (32) [right=of 22] {$(3,2)$};

\tikzset{my edge/.style = {thick,->,shorten >= 3pt,shorten <= 3pt}}

  \draw[my edge, dashed] [] (22) to [] (11);
\end{tikzpicture}
    \end{minipage}
 \[
\zeta_1(x) = (\emptyset, \{2\}), \qquad \zeta_2(x) = (\{1\}, \emptyset)
\]
\[
\frac{k\circ \zeta}{\gamma}(x) = (.1,.4)
\]
\[
\kappa(.1,.4) = (1,1)
\]   
    
    \caption{Continuing the example from Figure~\ref{fig:exampleRN} we construct the state transition graph. Above left is $X$ and right is $D$. For any $x$ in the shaded domain, the value of  $\frac{k\circ \zeta}{\gamma}(x)$ is constant and located in the lower left domain. This determines the value $\Phi^K(2,2) = (1,1)$ denoted with a dashed arrow; see \eqref{eq:commutativeDiagram}.} 
    \label{fig:exampleSTG}
\end{figure}
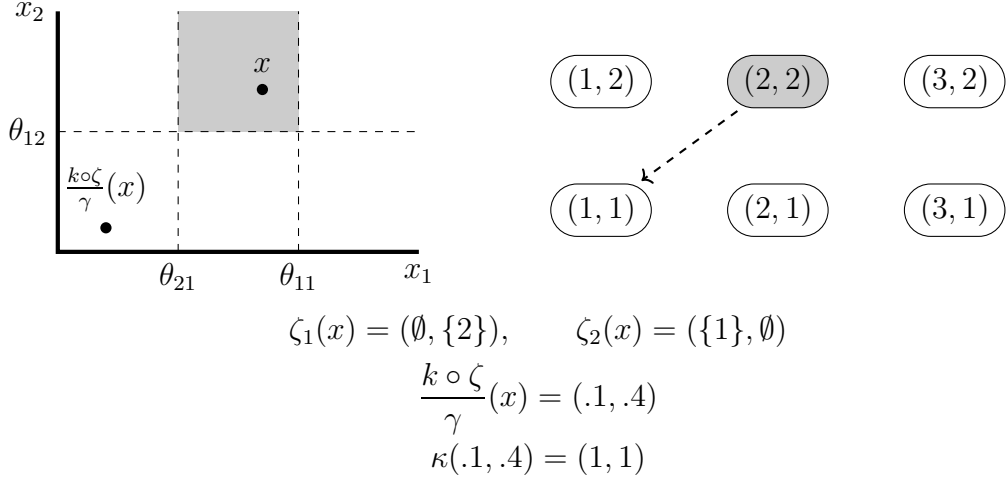

 Let  
  \[ D:= \prod_{i=1}^n \{0,1,\dots,b_i\} \]
   be a set of $n$ integer sequences that will be referred to as \textit{states}. Recall that $b_i$ is the number of targets of the node $i$ in \textbf{RN}, $b_i = |\boldsymbol \theta_i|$. We construct a function between the points of $X$ and the states in $D$. 
Let 
 $\kappa_i: X \to \{0,1,\dots,b_i\} $  be a map
\begin{equation}\label{eq:kappa} \kappa_i(x) = \ell \quad \mbox{ when }  x_i \in (\theta^{\ell}_{*i}, \theta^{\ell+1}_{*i} ) \end{equation}
where superscript $\ell$ indicates the $\ell$-th domain of the $x_i$ axis.
Collecting the maps $\kappa_i$  in a single map   we define $\kappa: X \to D$ to be the \textit{index assignment map} $\kappa = (\kappa_1, \ldots, \kappa_n)$. We say that $D$ \textit{indexes} the open domains of $X$. 
By construction, the index assignment map is constant on each domain in $X$. The map  $\kappa$ takes a vector $x \in X$ and assigns it to the state representing the domain of $X$ in which $x$ lies. 

The function 
\[ \frac{k \circ \zeta}{\gamma} = \left(\frac{k_1 \circ \zeta_1}{\gamma_1}, \ldots, \frac{k_n \circ \zeta_n}{\gamma_n}\right)\]
is a map $ \frac{k \circ \zeta}{\gamma}: X\to X$.  We define a discrete map $ \Phi^K: D \to D $ on the set of states $D$ by requiring that 
\begin{equation}\label{cd} \Phi^K \circ \kappa = \kappa \circ  \frac{k \circ \zeta}{\gamma}, \end{equation}
i.e. that the following diagram commutes
\begin{equation} \label{eq:commutativeDiagram}
  \begin{tikzcd}
   X \arrow{r}{\frac{k \circ \zeta}{\gamma} }    \arrow{d}{\kappa} & X \arrow{d}{\kappa} \\
  D  \arrow{r}{\Phi^K} & D
  \end{tikzcd}
\end{equation}

Note that the solution of (\ref{Ksystem}) with initial condition $x_0$  converges to the target point $\frac{k \circ \zeta}{\gamma} (x_0)$. The map $\Phi^K$ takes the state $d=\kappa(x_0)$ and assigns it to the state $\Phi^K(d)$ which contains the target point $\frac{k \circ \zeta}{\gamma} (x_0)$. In this way, the map $\Phi^K$ captures the behavior of solutions of (\ref{Ksystem}). 
It is important to note that the convergence of the solution starting at $x_0$  toward its  target point $\frac{k \circ \zeta}{\gamma} (x_0)$ is only valid while the solutions remain in the component of $X$ containing $x_0$; when they enter a neighboring domain, the target point will change.

To capture this behavior, we 
 define a {\it state transition graph} on states $d\in D$ that coarsely describes solutions of (\ref{Ksystem}). 
It represents the  \textit{asynchronous} update dynamics for the discrete valued function $\Phi^K$.
\begin{defn}[State transition graph]
 The \textit{state transition graph (STG)} is a directed graph with nodes $D$, where
two nodes $d, d' \in D$ are connected by a directed edge $d \to d'$, if and only if
\begin{enumerate}
\item either $d =d'$ and $\Phi^K(d) = d$; or 
\item $d$ and $d'$ differ in exactly one component, say $i$, and 
\begin{align*}
d'_i = d_i+1  &\mbox{ and } \Phi^K_i(d) >d_i, \text{ or } \\
d'_i = d_i-1  &\mbox{ and } \Phi^K_i(d) <d_i 
\end{align*}
\end{enumerate}
\end{defn}

We construct the state transition graph of our example network in Figures~\ref{fig:exampleSTG} and \ref{fig:exampleSTGpart2}.

\begin{figure}
\begin{tikzpicture}[ node distance = 1cm,every text node part/.style={align=left}]
\tikzset{my node/.style = {circle, draw, inner sep=1pt,minimum size=0pt}}

  \node[draw, rounded rectangle] (11) {$(1,1)$};
  \node[draw, rounded rectangle] (21) [right=of 11] {$(2,1)$};
  \node[draw, rounded rectangle] (31) [right=of 21] {$(3,1)$};
  \node[draw, rounded rectangle] (12) [above=of 11] {$(1,2)$};  
  \node[draw, rounded rectangle, fill=lightgray] (22) [right=of 12] {$(2,2)$};
  \node[draw, rounded rectangle] (32) [right=of 22] {$(3,2)$};

\tikzset{my edge/.style = {thick,->,shorten >= 3pt,shorten <= 3pt}}

  \draw[my edge] (22) to [] (12);
  \draw[my edge] (22) to [] (21);
  
  \draw[my edge, dashed] [] (22) to [] (11);

  \node[draw, rounded rectangle] (12') [right=3cm of 32] {$(1,2)$};  
  \node[draw, rounded rectangle] (11') [below= of 12'] {$(1,1)$};
  \node[draw, rounded rectangle] (21') [right=of 11'] {$(2,1)$};
  \node[draw, rounded rectangle] (31') [right=of 21'] {$(3,1)$};
  \node[draw, rounded rectangle] (22') [right=of 12'] {$(2,2)$};
  \node[draw, rounded rectangle] (32') [right=of 22'] {$(3,2)$};

\tikzset{my edge/.style = {thick,->,shorten >= 3pt,shorten <= 3pt}}

  \draw[my edge] [] (22') to [] (12');
  \draw[my edge] [] (22') to [] (21');
  \draw[my edge] [] (21') to [] (11');
  \draw[my edge] [] (12') to [] (11');
  \draw[my edge, loop left] [] (11') to [] (11');
  \draw[my edge] [] (32') to [] (31');
  \draw[my edge, loop right] [] (31') to [] (31');

  \coordinate[above=.5cm of 31]  (a) ;
  \coordinate[right=2cm of a]  (b) ;
  \draw[thick] (b) -- ++(0,1.3);
  \draw[thick] (b) -- ++(0,-1.3);
\end{tikzpicture}
    \caption{Left: The state transition graph is the asynchronous update dynamics, and so does not allow the diagonal transition; instead, we replace the dashed arrow with arrows capturing one-step adjacency. Right: the completed STG, where the process that was illustrated for state $(2,2)$ is repeated for each state.}
    \label{fig:exampleSTGpart2}
\end{figure}
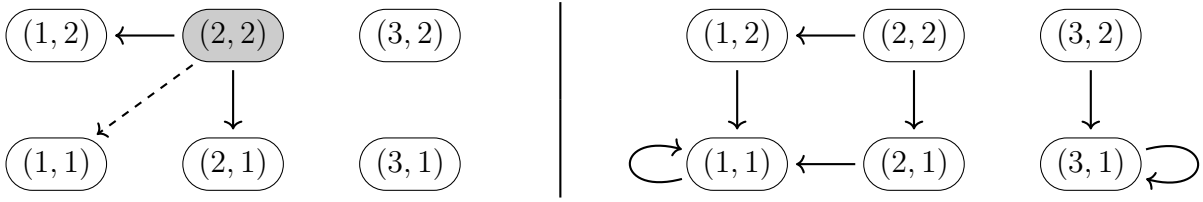


The number of maps $\Phi^K$ for a given \textbf{RN} is finite. This induces an equivalence relation over all collections $K$ satisfying the monotonicity condition in Definition~\ref{def:Kmono} that are consistent with the structure of \textbf{RN}.

\begin{defn}\label{defn:equivClasses}
For a given weighted regulatory network, we define an equivalence relation on the collection of all parameter sets $K$.
We set
\[ K \cong K' \iff  \Phi^{K}(d)  =  \Phi^{K'}(d) \quad \mbox{ for all } d \in D .\]
Notice that each equivalence class $[K]$ has a component structure composed of $n$ independent equivalence classes, $[K]=([K_1], \ldots, [K_n])$, one for each node $i \in V$ in the regulatory network. This is because the monotonicity assumption in Definition~\ref{def:Kmono} applies independently to each node.
\end{defn}

\begin{figure}
    \begin{minipage}{.48\textwidth}
    \centering
    \begin{tabular}{c|cc}
$y_1y_2$                & \multicolumn{1}{c}{$g_{11}$} & $g_{21}$ \\ \hline
00 & 0                           & 0       \\
01                      & 0                           & 0       \\
10                      & 1                           & 1       \\
11                      & 1                           & 1      
\end{tabular} \quad
\begin{tabular}{c|c}
$y_1$               & $g_{12}$ \\ \hline
0                     & 0       \\
1                      & 0
\end{tabular}

\vspace{12pt}

\begin{tikzpicture}[ node distance = 1cm,every text node part/.style={align=left}]
\tikzset{my node/.style = {circle, draw, inner sep=1pt,minimum size=0pt}}

  \node[draw, rounded rectangle] (11) {$(1,1)$};
  \node[draw, rounded rectangle] (21) [right=of 11] {$(2,1)$};
  \node[draw, rounded rectangle] (31) [right=of 21] {$(3,1)$};
  \node[draw, rounded rectangle] (12) [above=of 11] {$(1,2)$};  
  \node[draw, rounded rectangle] (22) [right=of 12] {$(2,2)$};
  \node[draw, rounded rectangle] (32) [right=of 22] {$(3,2)$};

\tikzset{my edge/.style = {thick,->,shorten >= 3pt,shorten <= 3pt}}

  \draw[my edge] [] (22) to [] (12);
  \draw[my edge] [] (22) to [] (21);
  \draw[my edge] [] (21) to [] (11);
  \draw[my edge] [] (12) to [] (11);
  \draw[my edge, loop left] [] (11) to [] (11);
  \draw[my edge] [] (32) to [] (31);
  \draw[my edge, loop right] [] (31) to [] (31);
  \end{tikzpicture}
    \end{minipage}
    \hfill\vline\hfill
    \begin{minipage}{.48\textwidth}
    \centering
    \begin{tabular}{c|cc}
$y_1y_2$                & \multicolumn{1}{c}{$g_{11}$} & $g_{21}$ \\ \hline
00						& 0                           & \cellcolor{lightgray}1      \\
01                      & 0                           & 0       \\
10                      & 1                           & 1       \\
11                      & 1                           & 1      
\end{tabular} \quad
\begin{tabular}{c|c}
$y_1$               & $g_{12}$ \\ \hline
0                      & 0       \\
1                      & 0
\end{tabular}

\vspace{12pt}

\begin{tikzpicture}[ node distance = 1cm,every text node part/.style={align=left}]
\tikzset{my node/.style = {circle, draw, inner sep=1pt,minimum size=0pt}}

  \node[draw, rounded rectangle] (11) {$(1,1)$};
  \node[draw, rounded rectangle] (21) [right=of 11] {$(2,1)$};
  \node[draw, rounded rectangle] (31) [right=of 21] {$(3,1)$};
  \node[draw, rounded rectangle] (12) [above=of 11] {$(1,2)$};  
  \node[draw, rounded rectangle] (22) [right=of 12] {$(2,2)$};
  \node[draw, rounded rectangle] (32) [right=of 22] {$(3,2)$};

\tikzset{my edge/.style = {thick,->,shorten >= 3pt,shorten <= 3pt}}

  \draw[my edge] [] (22) to [] (12);
  \draw[my edge] [] (22) to [] (21);
  \draw[my edge, dashed] [] (11) to [] (21);
  \draw[my edge] [] (12) to [] (11);
  \draw[my edge, loop right, dashed] [] (21) to [] (21);
  \draw[my edge] [] (32) to [] (31);
  \draw[my edge, loop right] [] (31) to [] (31);
\end{tikzpicture}

    \end{minipage}
    \caption{Left: the collection of MBFs corresponding to our example network and choice of $K$ in~\eqref{eq:Kexample}, and the associated state transition graph. Right: A collection of MBFs adjacent (in the parameter graph) to the collection on the left. The single change is highlighted in gray. The corresponding state transition graph is also pictured; differences caused by the shaded entry are shown as dashed edges}. 
    \label{fig:examplePGadj}
\end{figure}
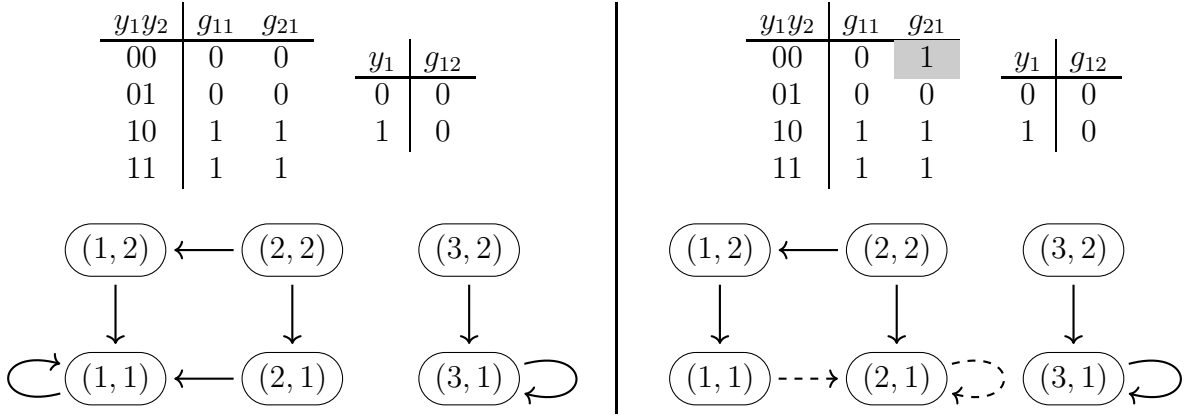

\subsection{Equivalence classes $[K]$ are collections of MBFs} \label{sec:PGBooleanFunction}
We now discuss the connection between equivalence classes $[K]$ and monotone Boolean functions. Each equivalence class $[K]$ is uniquely associated to a collection of $\prod_{i=1}^n b_i$ MBFs, where $b_i = |T(i)|$ is the number of targets of node $i$ in \textbf{RN}. We label these MBFs $g_{ki}$, one for each threshold $\theta_{ki}$ in the weighted regulatory network, and construct them below.

Let $[K]=([K_1], \ldots, [K_n])$ be an equivalence class, and consider an element $K_i \in [K_i]$, where $K_i = \{K_{i,A,B}\}_{A\subset S^+(i), B \subset S^-(i)}$ is a collection of constants for node $i$ in \textbf{RN}. 
Define a function 
\begin{align*}
\alpha_i : \B^{S(i)} \to  2^{S^+(i)} \times 2^{S^-(i)} \\
\end{align*}
where $\alpha_i(\vec y) = (A_i,B_i)$ for
\[ A_i= \{ j \in S(i) \mid y_j =1\} \cap  S^+(i), \qquad B_i= \{ j \in S(i) \mid y_j =1\} \cap  S^-(i) .\]
Here we use the standard multi-index notation $\B^{S(i)} = \{y_{j_1}y_{j_2}\dots y_{j_{m_i}} \;|\; j_k \in S(i), y_{j_k} \in \B\}$, i.e. elements of $\B^{S(i)}$ are Boolean strings of length $\abs{S(i)}$ indexed by elements of $S(i)$ in order. As an example, if $S(i)=\{2,4,5,7\}$, then $\vec y = y_2y_4y_5y_7$, where $y_i \in \B$.

Let 
$\{ \theta_{s_1 i}, \theta_{s_2 i}, \ldots \theta_{b_i i} \}$ be the $b_i$ thresholds associated to the $b_i$ targets of node $i$ in \textbf{RN}. 
With this assignment, we define positive  Boolean functions $g_{ki} : \B^{S(i)} \to \B$ as
\begin{equation*}
 g_{ki}(\vec y) = \begin{cases}
1 & \text{ when }  K_{i, \alpha_i(\vec y)} > \theta_{ki}\gamma_k \\
0 & \text{ when }  K_{i, \alpha_i(\vec y)} < \theta_{ki} \gamma_k
\end{cases}
\end{equation*}
and negative Boolean functions as
\begin{equation*}
 g_{ki}(\vec y) = \begin{cases}
0 & \text{ when }  K_{i, \alpha_i(\vec y)} > \theta_{ki}\gamma_k \\
1 & \text{ when }  K_{i, \alpha_i(\vec y)} < \theta_{ki} \gamma_k
\end{cases}
\end{equation*}

We observe that if $j \in S^+(i)$, then $g_{ki}$ will be positive in $x_j$, and if  $j \in S^-(i)$, then $g_{ki}$ will be negative in $x_j$. Therefore, any $g_{ki}$ constructed  in this way will be a monotone Boolean function  $g_{ki} \in \MBF(|S_i|)$.  As we show next, the collection 
\[G = \{ g_{ki} \;|\; i = 1,\dots,n, \; k \in T(i) \}\] 
is an equivalent representation of the equivalence class $[K]$.

\begin{prop} 
Fix a weighted regulatory network ${\bf RN}$ and thus sets $S^+(i), S^-(i)$, and $T(i)$, as well as the weights $\gamma_i$ and $\theta_{ji}$ for all $i=1,\ldots,n$, $j \in T(i)$. Then 
  \[ \Phi^K = \Phi^{K'} \iff G = G' .\]
  \end{prop}
  \begin{proof}
  Consider two different collections $K,K'$ associated to parameter assignment functions $k,k'$ respectively. Note that the set of states $D$, the function $\kappa: X \to D$ and the function $\zeta : X \to 2^{S^+(i)} \times 2^{S^-(i)}$ are uniquely determined by the weighted regulatory network  ${\bf RN}$. Therefore, $\Phi^K$ and $\Phi^{K'}$ differ only in the functions $k$ and $k'$:
 \[ 
    \Phi^K \circ \kappa = \kappa \circ \frac{k \circ \zeta}{\gamma} \qquad 
    \Phi^{K'} \circ \kappa = \kappa \circ \frac{k' \circ \zeta}{\gamma}.
\]
  Therefore it follows that
  \[ \Phi^K = \Phi^{K'} \iff \kappa \circ \frac{k}{\gamma} = \kappa \circ \frac{k'}{\gamma} \iff \kappa \circ \frac{k}{\gamma} \circ \alpha = \kappa \circ \frac{k'}{\gamma}\circ \alpha .\]
   This in turn leads to the equivalencies: 
       \begin{align*}
        &\iff \kappa(K_{i,\alpha_i(\vec y)}/\gamma_i) = \kappa(K'_{i,\alpha_i(\vec y)}/\gamma_i) \; \forall i = 1,\dots,n \text{ and }\forall \vec y \in \B^{S(i)}\\
        &\iff K_{i,\alpha_i(\vec y)}/\gamma_i, K'_{i,\alpha_i(\vec y)}/\gamma_i \in (\theta_{*i}^\ell,\theta_{*i}^{\ell +1}) \text{ for some } \ell \\
        &\iff K_{i,\alpha_i(\vec y)}, K'_{i,\alpha_i(\vec y)} \in (\theta_{*i}^\ell\gamma_i,\theta_{*i}^{\ell +1}\gamma_i) \text{ for some } \ell \\
        &\iff \big(K_{i,\alpha_i(\vec y)} > \theta_{ji}\gamma_i \iff K'_{i,\alpha_i(\vec y)} > \theta_{ji}\gamma_i\big) \;\forall j \in T(i)\\
        &\iff \big(g_{ji} \in G \iff g_{ji} \in G'\big).
        \end{align*}
\end{proof}

For an example, compare the collection $K$ in~\eqref{eq:Kexample} for the weighted regulatory network in Figure~\ref{fig:exampleRN} to the equivalent collection of three monotone Boolean functions in the left panel of Figure~\ref{fig:examplePGadj}.

\begin{figure}
    \centering
    \begin{tikzpicture}[ node distance = .7cm]
\tikzset{my node/.style = {circle, draw, inner sep=1pt,minimum size=0pt}}
	\node[draw, rounded rectangle, thick] (A) {
		\begin{tikzpicture}[ node distance = .5cm,every text node part/.style={align=left}]
			\tikzset{my node/.style = {circle, draw, inner sep=2pt,minimum size=0pt}}
				\node[my node] (1) {$0$};
				\node[my node] (2) [right=of 1] {$1$};
				\draw[-] [] (1) to (2);
		\end{tikzpicture}
  };
	\node[draw, rounded rectangle, thick] (B) [right=of A] {
		\begin{tikzpicture}[ node distance = .5cm]
			\tikzset{my node/.style = {circle, draw, inner sep=2pt,minimum size=0pt}}
				\node[my node, fill=darkgray] (1) {$0$};
				\node[my node] (2) [right=of 1] {$1$};
				\draw[-] [] (1) to (2);
		\end{tikzpicture}
  };
	\node[draw, rounded rectangle, thick] (C) [right=of B] {
		\begin{tikzpicture}[ node distance = .5cm]
			\tikzset{my node/.style = {circle, draw, inner sep=2pt,minimum size=0pt}}
				\node[my node, fill=darkgray] (1) {$0$};
				\node[my node, fill=darkgray] (2) [right=of 1] {$1$};
				\draw[-] [] (1) to (2);
		\end{tikzpicture}
  };
  \draw[-, thick] [] (A) to [] (B);
  \draw[-, thick] [] (B) to [] (C);
    \node[] (D) [left=of A] {$\times$};
    \node[draw, rounded rectangle, thick] (E) [left=of D] {20 node graph};
\end{tikzpicture}
    \caption{Continuing the example, we show the parameter graph for the network in Figure~\ref{fig:exampleRN}. The parameter graph takes the form of a product graph, with one factor for each node in $\bRN$. The factor associated to node $1$, which has two inputs and two outputs in \textbf{RN}, is the 20 node graph on the left. It is isomorphic to the graph shown in Appendix~\ref{problemCubes}, Figure~\ref{fig:megaPG}. The factor associated to node $2$, which has one input and one output in \textbf{RN}, 
    is shown on right. Each parameter node in the factor on the right contains the associated monotone Boolean function, where gray shading means $g_{12} = 0$, and similarly white shading implies $g_{12} = 1$.}
    \label{fig:exampleParameterGraph}
\end{figure}
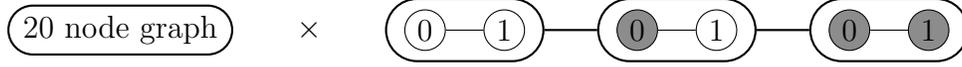

\subsection{Parameter Graph}\label{sec:PG}
The fact that we can view equivalence classes $[K]$ as collections of monotone Boolean functions allows us to organize the equivalence classes $[K]$ into a graph, called the {\it parameter graph (PG)}, where each node is associated to an equivalence class. Let $[K]$ and $[K']$ be different equivalence classes, with associated collections of MBFs $G= \{g_{ki}\}$ and $G'=\{g'_{ki}\}$. 
The nodes $[K]$ and $[K']$ will be connected by an edge, if, and only if, there is $i\in \{ 1, \ldots, n\}$, $k \in T(i)$, and $\vec y \in \B^{S(i)}$ such that  
\[g_{ki}(\vec y) \neq g_{ki}'(\vec y)\]
\[g_{ki}(\vec z) = g_{ki}'(\vec z) \quad \forall \vec z \neq \vec y\]
\[g_{\ell j} = g_{\ell j}' \text{ whenever } j \neq i \text{ or } \ell \neq k. \]
In other words, there is exactly one MBF whose value  differs on one input.
An example of a single adjacency is shown in Figure~\ref{fig:examplePGadj} and the parameter graph for our running example is shown in Figure~\ref{fig:exampleParameterGraph}.

\subsection{Representative networks}

It will be convenient to consider a subset of weighted regulatory networks \textbf{RN} with $\gamma_i = 1$ for all $i$. It turns out that the class of weighted regulatory networks with this property  exhibits the same parameter graphs  with the same collection of state transition graphs as the collection of weighted graphs with general positive decay rates $\gamma=(\gamma_1, \ldots, \gamma_n)$.

To see this, fix a set $K$ and its parameter assignment function $k$ from Definition~\ref{def:Kmono}. Consider a weighted regulatory network \textbf{RN} with the collections of sources $S^+(i)$, $S^-(i)$, targets $T(i)$, decay rates $\gamma=(\gamma_1, \ldots, \gamma_n)$, and thresholds $\{\theta_{ji}\}$. Consider a network $\widehat{\mbox{\textbf{RN}}}$ with the same collection of sources  $S^+(i)$, $S^-(i)$ and targets $T(i)$, but with all decay rates set to $\hat \gamma_i = 1$ and the thresholds set to $\{\hat \theta_{ji} = \gamma_i \theta_{ji}\}$. 

The threshold assignment induces a bijection $ x \mapsto \hat x$ with $\hat x = \gamma x$,  from $X$ to $\hat X$:
  \[\hat X = \R^n_+ \setminus \{ \hat x_i = \hat\theta_{ji} \;|\; i = 1,\dots,n, \quad j \in T(i)\}.\] 
  The key observation is that 
\begin{equation}\label{eq:obs}
    \hat x_i \in \left(\gamma_i\theta_{*i}^\ell,\gamma_i\theta_{*i}^{\ell+1}\right) \; \Leftrightarrow \; x_i = \frac{\hat x_i}{\gamma_i} \in \left(\theta_{*i}^\ell,\theta_{*i}^{\ell+1}\right).
\end{equation}
This allows us to conclude that the activity functions $\zeta$ and $\hat \zeta$ (defined in~\eqref{eq:zeta}) satisfy $\hat \zeta(\hat x) = \zeta(\hat x / \gamma) =  \zeta(x)$, which leads to
  \[  \theta_{ji} < \frac{k_i \circ \zeta}{\gamma_i} (x)< \theta_{si} \]
  for some $j,s \in T(i)$, if and only if
  \[ \hat \theta_{ji} < k_i \circ \hat \zeta (\hat x)< \hat \theta_{si} .\]
  In other words, the following diagram commutes:
\begin{equation*} 
  \begin{tikzcd}
  X \arrow{r}{\frac{k \circ \zeta}{\gamma} }    \arrow{d}{\gamma x} & X \arrow{d}{\gamma x} \\
  \hat X \arrow{r}{ k \circ \hat \zeta        } & \hat X 
  \end{tikzcd}
  \end{equation*}

Since the underlying network topology is the same between two weighted networks \textbf{RN} and $\widehat{\mbox{\textbf{RN}}}$, the discrete states of the state transition graph are the same $D = \hat D$. Using~\eqref{eq:obs} again, we conclude that the index assignment functions $\kappa$ and $\hat \kappa$ from~\eqref{eq:kappa} satisfy the following:
\[ \hat \kappa(\hat x) = \ell \; \Leftrightarrow \; \kappa(x) = \ell. \]
Recalling that $\Phi^K \circ \kappa = \kappa \circ (k \circ \zeta)/\gamma$ from~\eqref{cd}, we see that 
\[  \hat \Phi^K \circ \hat \kappa(\hat x) = \Phi^K \circ \kappa(x).\]
This means that the state transition graphs are identical under $K$ applied to \textbf{RN} and $\widehat{\mbox{\textbf{RN}}}$.

We conclude that by  considering  the restricted class of weighted regulatory networks  with $\gamma_i=1$ for all $i$ we will recover the same set of state transition graphs as the general system.  Therefore we we will  assume $\gamma_i=1$ from now on, and  we will write $g_{ki}$ as 
\begin{align} \label{eq:g}
 g_{ki}(\vec y) & = \begin{cases}
1 & \text{ when }  K_{i, \alpha_i(\vec y)} > \theta_{ki} \\
0 & \text{ when }  K_{i, \alpha_i(\vec y)} < \theta_{ki} 
\end{cases} \nonumber \\
 & \mbox{ or } \\
 g_{ki}(\vec y) &= \begin{cases}
0 & \text{ when }  K_{i, \alpha_i(\vec y)} > \theta_{ki} \\
1 & \text{ when }  K_{i, \alpha_i(\vec y)} < \theta_{ki}
\end{cases} \nonumber
.
\end{align}

\subsection{Differential equations from monotone Boolean functions}
In Definition~\ref{def:Ksystem} we associated an ordinary differential equation to a weighted regulatory network {\bf RN}. 
A more explicit way to do so is due to ~\cite{glass:kaufman:72,glass:kaufman:73,Thomas1973}.
Again consider a weighted regulatory network \textbf{RN} with nodes $i \in V$ summarizing regulatory activity for continuous variables $x_i$. 
Assume that regulation  of $x_i$ by its regulatory input  $x_j$ switches  abruptly  at the real-valued threshold $\theta_{ij}$ from \textbf{RN}, written as one of two maps
\[ \sigma^+_{ij}(x_j) = \left \{ \begin{array}{cc} 
  1 & x_j> \theta_{ij}\\
  0 & x_j < \theta_{ij}
  \end{array} \right. 
  \quad 
  \sigma^-_{ij}(x_j) = \left \{ \begin{array}{cc} 
  0 & x_j > \theta_{ij}\\
  1 & x_j < \theta_{ij}
  \end{array} \right. 
\]
  whenever $j \in S(i)$ is a source of node $i$.
  We write $\sigma_i = (\sigma_{i s_1},\dots,\sigma_{i s_{m_i}})$, where $s_j \in S(i)$, $m_i = |S(i)|$, $\sigma_{ik} = \sigma^+_{ik}$ whenever $k \in S^+(i)$, and $\sigma_{ik} = \sigma^-_{ik}$ whenever $k \in S^-(i)$. In other words, $\sigma^+$ models an activating input and $\sigma^-$ a repressing input.

  We shall again assume  that any two thresholds $\theta_{ji}$ and $\theta_{ki}$ are distinct for variable $x_i$. Also as before, these thresholds $\{\theta_{ij} \;|\; i = 1,\dots,n, j \in S(i)\}$ divide $\bbR^n_+$ into an open rectangular grid $X$. In addition, assume that for very node $i\in V $ there is an associated Boolean function $f_i: \B^{S(i)} \to \B$ that converts inputs of the node $i$ to the new state of node $i$. Then we consider the  following system of ODEs on $X$:
\begin{equation} \label{switching}
\dot{x}_i = -x_i + f_i(\sigma_i(x)) 
\end{equation}

\section{Algebraic switching systems}

The system of ODEs~\eqref{switching} has no continuous parameters. In order to introduce such parameters and allow comparison with K-systems we parameterize both the domain and the range of each function 
 $f_i : \B^n \to \B$. 
 To parameterize the domain we replace in the definition of $\sigma_{ij}$ the Boolean values $0<1$ by continuous, non-negative, real parameter values $L_{ij}<U_{ij}$.
  To capture the sign along the network edges, we again consider two types of $\sigma_{ij}$ functions
 \[ \sigma^+_{ij}(x_j) = \left \{ \begin{array}{cc} 
 U_{ij} & x_j> \theta_{ij}\\
 L_{ij} & x_j < \theta_{ij}
 \end{array} \right.  \quad  
 \sigma^-_{ij}(x_j) = \left \{ \begin{array}{cc} 
 L_{ij} & x_j> \theta_{ij}\\
 U_{ij} & x_j < \theta_{ij}
 \end{array} \right. .
 \]

We introduce the \textit{interaction function} $\Lambda_i$ as a real-valued replacement for the function $f_i$. All interaction functions will be algebraic expressions over the real numbers using addition and multiplication.
In this notation, Equation~\eqref{switching} reads
 \[
 \dot{x}_i = -x_i + \Lambda_i(\sigma_i(x)) ,
 \]
and we refer to it as a \textit{switching system}, as in \cite{Cummins16,Gedeon18}. 

For every $x \in X$, the composition $\Lambda_i(\sigma_i(x))$ assigns a real number that is a combination of numbers $\{L_{ij} \mbox{ or } U_{ij} \;|\; j \in S(i)\}$, where for each $j$ only one of $L_{ij} \mbox{ or } U_{ij}$ enters the function $\Lambda_i$.  This combination is constant on each domain in $X$.  For monotone functions $\Lambda_i$, the image of $\Lambda_i(\sigma_i(x))$ is a set $K$  that satisfies the monotonicity assumption in Definition~\ref{def:Kmono}.
Therefore the  switching system with monotone $\Lambda_i$ for all $i$ is a K-system (\ref{Ksystem}) and therefore  gives rise to  a parameter graph and to a state transition graph for each parameter node.

In this paper we consider three basic algebraic forms of functions $\Lambda_i$. The set of linear ($\Sigma$) $\Lambda$ functions is given by 
 \[
 \boldsymbol \Sigma{(n)} := \left\{\Lambda_i: \R^n_+ \to \R_+ \mid V \subseteq \{1,\dots n\}, \Lambda_i(z_1, \ldots, z_n)  =  \sum_{i \in V} z_i\right\} \ .
 \]
 The products of sums ($\Pi\Sigma$) $\Lambda$ functions are
  \[
\boldsymbol \Pi\boldsymbol\Sigma{(n)} := \left\{\Lambda_i: \R^n_+ \to \R_+ \mid  \Lambda_i(z_1, \ldots, z_n)   = \prod_{W_{k}} \left (\sum_{i \in W_{k}} z_i \right) \right\} \ ,
 \]
where the sets $W_k$ partition $S(i)$.
 The sums of products of sums ($\Sigma\Pi\Sigma$) $\Lambda$ functions are
  \begin{equation} \label{big:class}
 \boldsymbol \Sigma \boldsymbol\Pi\boldsymbol\Sigma{(n)}  := \left\{\Lambda_i: \R^n_+ \to \R_+ \mid \Lambda_i(z_1, \ldots, z_n)   =  \sum_{W_k} \left( \prod_{V_{k,j}} \left (\sum_{i \in V_{k,j}} z_i \right) \right)\right\} \ ,
 \end{equation}
where the disjoint union of $V_{k,j}$ is $W_k$. As before, the sets $W_k$ partition $S(i)$, and the sets $V_{k,j} $ partition the set $W_k$. 
Observe that these classes of functions contain progressively more functions, i.e.\ for $n\geq 3$,
 \[ \boldsymbol \Sigma{(n)} \subsetneq \boldsymbol \Pi\boldsymbol\Sigma{(n)}  \subsetneq \boldsymbol \Sigma \boldsymbol\Pi\boldsymbol\Sigma{(n)} .\] 

  The restriction of the class of functions $\Lambda$ to a product of sums ($\Pi\Sigma$)  goes back to at least Snoussi~\cite{Snoussi89}, and  was used extensively in the development of DSGRN (Dynamic Signatures Generated by Regulatory Networks)~\cite{Cummins16,Gedeon18,Cummins2017b,Kahrl18,EMT}.   The main contribution of the DSGRN approach is the definition and explicit construction of a parameter graph (Section \ref{sec:PG}) in terms of inequalities in the input combinations of $\{L_{ij}, U_{ij} \;|\; j \in S(i) \}$ and thresholds  $\{\theta_{k i} \; |\; k \in T(i)\}$ given a collection of  nonlinearities in $\Pi \Sigma$.

 As we have shown in Section~\ref{sec:PGBooleanFunction}, each component $[K_i]$ of the parameter node $[K] = ([K_1], \ldots, [K_n])$ is equivalent to a collection of monotone Boolean functions, one for each edge in the regulatory network. When a node $i$ has a single target, then there is only one MBF associated to node $i$, namely $g_{ji}$, where $j$ is the sole target of $i$. In the case of multiple targets, $|T(i)| > 1$, there is a collection of $|T(i)|$ MBFs for node $i$.
We consider a single component $[K_i]$, first where node $i$ has a single target and second where node $i$ has more than one target. We ask which such singletons or collections of monotone Boolean functions can be associated to a parameter node in the parameter graphs of $K$-, $\Sigma$-, $\Pi \Sigma$- and $\Sigma \Pi \Sigma$-systems.

 \begin{defn}\label{A:realizable}
 
We say that a  monotone Boolean function $h:\B^n \to \B$ is \emph{$\ast$-realizable}, where $\ast$ can stand for $K$, $\Sigma$, $\Pi \Sigma$ or $  \Sigma \Pi \Sigma$, if there exists a regulatory network \textbf{RN} with a node $j$ with a single target $\ell$ and weight $\theta_{\ell j}$ and a parameter node $[K]=([K_1], \ldots, [K_n])$ for the $\ast$-system, such that the Boolean map $g_{\ell j}$ that corresponds to  $[K_j]$ is $h$.

We say that a $k$-tuple of monotone Boolean functions $h_1, h_2,\dots, h_k: \B^n \to \B$ is  \emph{$\ast$-jointly realizable} if there exists a regulatory network \textbf{RN} with a node $j$ with $k$ targets $\ell_1, \dots , \ell_k$ and weights $\{ \theta_{\ell_i j} \}$ and a parameter node $[K]=([K_1], \ldots, [K_n])$ for the $\ast$-system such that the collection of Boolean maps $g_{\ell_1j}, g_{\ell_2j}, \ldots, g_{\ell_kj}$ that corresponds to  $[K_j]$   are the maps 
$h_1, h_2,\dots, h_k$ respectively. 

\end{defn}

 \begin{rem} \label{rem:realizabilitySpecialCases}
  We note that $\Sigma$-realizability is a special case of $\Pi\Sigma$-realizability, which is in turn a special case of $\Sigma\Pi\Sigma$-realizability, which is in turn a special case of $K$-realizability. These observations rely on the fact that $\boldsymbol \Sigma \subset \boldsymbol \Pi\boldsymbol\Sigma  \subset \boldsymbol \Sigma \boldsymbol\Pi\boldsymbol\Sigma$ and that the images of any monotone $\Lambda$ functions give rise to $K$-systems.
  \end{rem}
 
 The definition of realizability~\ref{A:realizable}  uses arbitrary functions $h \in \MBF$. In some cases, it will be convenient to assume $g_{ki} \in \MBF^+(m_i)$, instead of the weaker condition $g_{ki} \in \MBF(m_i)$, where recall $m_i = |S(i)|$ is the number of sources of $i$. This is achieved via a coordinate change. For each $j \in S(i)$, define the function $\beta_j^{(i)} : \B \to \B$ as
 \begin{equation}\label{eq:beta}
 \beta_j^{(i)} (b) = \begin{cases}
 b &\text{if } j\in S^+(i) \\ 
 1-b &\text{if } j\in S^-(i)
 \end{cases}
 \end{equation}
 Then define
 \[
 \beta^{(i)} : \B^{m_i} \to \B^{m_i}, \quad \beta^{(i)} = (\beta_1^{(i)}, \ldots, \beta_{m_i}^{(i)} )
 \]
 component-wise. We observe that $g_{ki} \circ \beta^{(i)} \in \MBF^+(m_i)$ and that  $\beta^{(i)}$ is an involution i.e. $\beta^{(i)}\circ \beta^{(i)} = \mathrm{Id}$. We will use the notation $\beta$ for a function where we do not specify the network node identity $i$.  

 Using the coordinate change $\beta$, we can see that $h:\B^{m_i} \to \B$ is $\ast$-realizable if and only if, $f:=h\circ \beta \in \MBF^+(m_i)$ is a positive Boolean function and is also $\ast$-realizable, via the collection $K'$ defined as 
 \[
 K'_{i,A', \emptyset} = K_{i,A,B}
 \]
 where
 \[
 A' = A \cup B,
 \]
and network $\bRN'$ which is the same as $\bRN$ except that all edges are now activating. Likewise, if $h_1, h_2,\dots, h_k: \B^n \to \B$ is  $\ast$-jointly realizable if, and only if,  $f_1= h_1\circ\beta^{(1)}, \dots, f_k:=h_k \circ  \beta^{(k)}\in \MBF^+(n)$ are positive Boolean functions and are also $\ast$-jointly realizable.   Therefore it is sufficient to consider in proofs only  positive Boolean functions. 

  The central question that we pose  in this paper is to ask how much restriction the algebraic forms $\Sigma$, $\Pi \Sigma$, and $  \Sigma \Pi \Sigma$ impose on the richness of the potential dynamics of the switching system. 
 We will interpret the number of  $k$-tuples of MBFs that can be represented in the parameter graph  as the richness of that particular class of switching systems. 
 This question generalizes and extends a classical problem of determining when a monotone Boolean function is a threshold function. 
 
\begin{defn}\label{def:threshold}\cite{CH11}
A Boolean function $f: \B^n \to \B$, is called a \textit{threshold function} (or a \textit{linearly separable function}) if there exist real numbers $a_1,\dots,a_n \in \R$ and a threshold $\theta \in \bbR$ such that for all $\vec y = (y_1,\dots,y_n) \in \B^n$,
\begin{equation}
f(\vec x) = \begin{cases} 1 & \text{if }  \sum_{j=1}^n a_jy_j > \theta \\
0 & \text{otherwise}
\end{cases}.
\end{equation}
The $(n+1)$-tuple $(a_1,a_2,\dots,a_n,\theta)$ is called a \textit{(separating) structure} of $f$.
\end{defn}

As we will see later in Lemma~\ref{lem:SigmaIsThreshold}, any monotone Boolean function $f$ is a threshold function if and only if $f$ is $\Sigma$-realizable, i.e. representable in the parameter graph of a $\Sigma$ system. Framed in terms of threshold functions, determining which $f$ functions are $\Sigma$-realizable is then equivalent to determining which MBFs are indeed threshold functions. 

This is a classical problem in the Boolean literature. Paull \cite{paull1960boolean} showed that monotonicity is a necessary condition for a Boolean function to be a threshold function. As shown by Chow \cite{chow1961boolean} and Elgot \cite{elgot1961truth}, a Boolean function is a threshold function if, and only if,  it is \textit{assumable}, where assumable was defined by Winder \cite{winder1962threshold}. An algorithm for determining whether a Boolean function is a threshold function was given by Peled and Simeone \cite{peled1985polynomial}. Their  algorithm will produce $a_1,\dots,a_n$ in the sense of Definition~\ref{def:threshold} if the Boolean function is indeed a threshold function. An algorithm for producing additional linearly separable Boolean functions and further characterization of threshold functions was given in Rao and Zhang \cite{rao2016characterization}. The number of threshold functions for $n\leq 8$ was found in \cite{muroga1970enumeration}, and extended to $n=9$ by Gruzling \cite{GruzlingN}.

Note that $f$ is a threshold function if the set of points in $\B^n \subset \bbR^n$ at which $f$ attains value $1$ is linearly separable from those points where $f$ attains value $0$.  Following this connection, Pantovic et al. \cite{pantovic2014number}, Zunic \cite{zunic2004encoding}, and Wang and Williams \cite{wang1991threshold} all examined partitions of sets of points with surfaces that are not necessarily hyperplanes. This is intimately related to the questions of $\Pi\Sigma$- and $\Sigma\Pi\Sigma$-realizability, i.e. which monotone Boolean functions are representable in a parameter graph of a $\Pi\Sigma$- vs.  $\Sigma\Pi\Sigma$- systems.

The parameter graph node $[K]=([K_1], \ldots, [K_n])$ represents $n$ collections of monotone Boolean functions, where each  $[K_j]$ corresponds to $b_j = |T(j)|$ MBFs, the number of targets of $j$ in \textbf{RN}. Not every collection 
of $b_j$ monotone Boolean functions is $*$-jointly realizable for the algebraic classes we consider. We introduce the idea of considering multiple Boolean functions simultaneously and asking which of them are  *-jointly realizable, i.e. realizable in a node in the parameter graph of a $K$, $\Sigma$, $\Pi\Sigma$ or $\Sigma\Pi\Sigma$ system.

\section{Realizability Results}

\begin{table}[htp]
\centering
\begin{tabular}{|c|c c c c c c c|} \hline
$n$ & \multicolumn{7}{c |}{$\ast$-Joint Realizability of $f\prec g$} \\ \hline
1 & $\Sigma$ & $=$ & $\Pi\Sigma$ & $=$ & $\Sigma\Pi\Sigma$ & $=$ & K \\
2 & $\Sigma$ & $=$ & $\Pi\Sigma$ & $=$ & $\Sigma\Pi\Sigma$ & $=$ & K \\ 
3 & $\Sigma$ & $\subsetneq$ & $\Pi\Sigma$ & $\subsetneq$ & $\Sigma\Pi\Sigma$ & $=$ & K \\
$\geq$ 4 & $\Sigma$ & $\subsetneq$ & $\Pi\Sigma$ & $\subsetneq$ & $\Sigma\Pi\Sigma$ & $\subsetneq$ & K \\ \hline
\end{tabular}
\caption{Summary of the results. The $n$ is the number of inputs for each of the pair of MBF, where $f \prec g$ means that the truth set of $f$ is a subset of truth set of $g$ (Definition~\ref{def:prec}). The (in)equalities express realizability relations among categories of functions (see text).  Row 1 and 2 are a result of Proposition~\ref{first:two}. 
Row 3 is a result of Subsections~\ref{subsection:SigmaSubsetPiSigma}, \ref{subsection:PiSigmaSubsetSigmaPiSigma}, and \ref{SigmaPiSigmaEqualsK}. Row 4 is a result of Subsection~\ref{subsection:SigmaPiSigmaSubsetK}.} \label{tableOfResults}
\end{table}

The main results  are summarized in Table~\ref{tableOfResults}, where we consider pairs of $*$-jointly realizable MBFs. 
The equality sign between two categories indicates that whenever a pair of MBFs  with given set of inputs $n$ (row index) is realizable in one category, it is also realizable in the other category. The strict subset relationship shows that any pair of functions realizable in the smaller category is also realizable in the larger category, and, furthermore,  there is a pair of Boolean functions $f \prec g$ that is realizable in the larger category that is not realizable in the smaller category.

\subsection{$K$-realizability}

The main goal of this section is to prove  the following two results: (1) any monotone Boolean function is $K$-realizable, and (2) any $k$-tuple $f_1\prec\ldots \prec f_k$ of MBFs is $K$-jointly realizable. Therefore $K$-systems are general enough to represent  any collection of monotone Boolean functions.  

\begin{defn}
For a Boolean function $f: \B^n \to \B$, we define the \textit{truth set} of $f$ as 
\begin{equation*}
\Truth(f) := \{\vec y \in \B^n \mid f(\vec y) = 1\}.
\end{equation*}
Similarly we define the \textit{false set} of $f$ as 
\begin{equation*}
\False(f) := \{\vec y \in \B^n \mid f(\vec y) = 0\}.
\end{equation*}
For $U \subseteq \B^n$ any subset, we will also use the notation 
\begin{equation*}
\Truth(f)|_U := \{\vec y \in U \mid f(\vec y) = 1 \}, \quad 
\False(f)|_U := \{\vec y \in U \mid f(\vec y) = 0 \}.
\end{equation*}
\end{defn}

We start our discussion of $K$-realizability and $K$-joint realizability by proving two results relating $K$-realizability to the existence of what we call a \textit{realizing function}.

\begin{defn}\label{def:Rmono}
    The \textit{positive monotonicity assumption} for a function $\cR: \B^n \to \R_+$ is the following: for all $j\in\{1,\dots,n\}$, for all $\vec y \in \B^n$ with $y_j = 0$  
\[
  \cR(\vec y) \leq \cR(\vec y + \hat e_j) .
  \]
  \end{defn}

\begin{thm}\label{thm:K}
  ~\\
\begin{enumerate}
\item $f\in \MBF^+(n)$ is $K$-realizable if, and only if, there exist
\begin{itemize}
    \item a weighted regulatory network \textbf{RN} with a node $i$ with only one target $j$ and a weight $\theta_{ji}$ and 
    \item a function $\cR^{(i)}$ such that $\cR^{(i)}: \B^n \to \R_+$ satisfies the positive monotonicity assumption and
\[
f(\vec y) = \begin{cases}
1 & \text{ if } \cR^{(i)}(\vec y) > \theta_{ji} \\
0 & \text{ if } \cR^{(i)}(\vec y) < \theta_{ji}
\end{cases} \, .
\]
\end{itemize}

\item A $k$-tuple of MBFs $f_1, f_2,\dots, f_k\in \MBF^+(n)$ is $K$-jointly realizable if, and only if, there exist 
\begin{itemize}
    \item a weighted regulatory network \textbf{RN} with a node $i$ with $k$ targets $\ell_1,\dots,\ell_k$ and weights $\{\theta_{\ell_ji}\}$ and
    \item a function $\cR^{(i)}$ such that $\cR^{(i)}: \B^n \to \R_+$ satisfies the positive monotonicity assumption and for all $j \in \{1,\dots,k\}$, $f_j$ can be expressed as
\begin{equation} \label{def2:K}
f_j(\vec y) = \begin{cases}
1 & \text{ if } \cR^{(i)}(\vec y) > \theta_{\ell_ji} \\
0 & \text{ if } \cR^{(i)}(\vec y) < \theta_{\ell_ji}
\end{cases} \, .
\end{equation}
\end{itemize}

\end{enumerate}
\end{thm}

\begin{proof}
Using Definition~\ref{A:realizable}, Equation~\eqref{eq:g}, and setting 
\[\cR^{(i)}(\vec y) \equiv K_{i,\alpha_i(\vec y)}, \quad \vec y \in \B^{n},\]
 the theorem follows. It  remains to note that the positive monotonicity assumption on $\cR^{(i)}$ induces the (positive) monotonicity condition on $K_i$ from Definition~\ref{def:Kmono}. Likewise, when $f_1,\dots,f_k \in \MBF^+(n)$, then $S(i) = S^+(i)$ and $K_i$ satisfying the monotonicity condition in Definition~\ref{def:Kmono} implies that $\cR^{(i)}$ must satisfy the positive monotonicity condition.
\end{proof}

\begin{defn}
  If Theorem~\ref{thm:K} is satisfied, then the pair ($\cR^{(i)}$,\textbf{RN}) is called a \emph{realizing function} and \emph{realizing network} for $f_1,\dots,f_k$, respectively. 
\end{defn}

Theorem~\ref{thm:K} shows that  $K$-systems can be thought of  as arising from monotone Boolean functions via realizing functions $\cR^{(i)}$, one for each node $i$ in a realizing network \textbf{RN}. In the following, we will restrict our focus to a single node in {\bf RN}  and drop the superscript.

\begin{defn} \label{def:prec}
  For two Boolean functions $f, g: \B^n \to \B$, we say $f$ \textit{implies} $g$ and write $f \prec g$ if, and only if, $\Truth(f) \subseteq \Truth(g)$.
  \end{defn}
    
Now we prove the main result of this section, namely that all  $k$-tuples of MBFs that are linearly ordered $f_1 \prec f_2 \prec \ldots \prec f_k$  are $K$-(jointly) realizable for all $k\geq 1$.

\begin{thm}\label{thm:KCompatible}
  ~
  \begin{enumerate}
    \item $f\in \MBF^+(n)$ if and only if $f$ is $K$-realizable.
  \item A collection $f_1, \ldots, f_b\in \MBF^+(n)$ of monotone Boolean functions has a linear order $f_1 \prec f_2 \prec \ldots \prec f_b$ if and only if it is  $K$-jointly realizable.
  \end{enumerate}
  \end{thm}

  \begin{proof}
    Since realizability is a special case of joint realizability and since a single MBF trivially has an order, it is sufficient to prove the second point.
    
    $(\Rightarrow)$ Let  $\cR(\vec y) := \sum_{j=1}^b f_j(\vec y)$. $\cR$ satisfies the positive monotonicity assumption of Definition~\ref{def:Rmono}, since if $y_i = 0$ for some $\vec y$, then $\cR(\vec y) \leq \cR(\vec y + \hat e_i)$ by summation and the positivity of $f_j$. Now  for  each $j \in \{1,\dots,b\}$, let $\theta_j = b - j + \frac{1}{2}$. Suppose $\vec y \in \Truth(f_j)$. Then $\vec y \in \Truth(f_k)$ for $k = j,\dots,b$, since $\Truth(f_j) \subseteq \Truth(f_k)$ by the $\prec$ relationship. So $\cR(\vec y) = b-j+1$ and we have $\theta_j = b - j + 1/2 < \cR(\vec y)$ as desired. Then the  function $\cR$ and thresholds $\theta_1, \ldots, \theta_b$ satisfy the assumptions of Theorem~\ref{thm:K}. 

    $(\Leftarrow)$ 
    Given the $b$ thresholds, establish the indexing $\theta_{b} < \dots < \theta_{1}$ using the order of $\R$. Then for any $i < j$, we have $\theta_{i} > \theta_{j}$ and $(\theta_{i}, \infty) \subseteq (\theta_{j}, \infty)$. Given the realizing function $\cR$, define a collection of $b$ positive monotone Boolean functions by
    \[ \Truth(f_{i}) =  \cR^{-1}(\theta_{i}, \infty). \]
    Then by construction 
    \[
    f_i(\vec y) = \begin{cases}
1 & \text{ if } \cR(\vec y) > \theta_i \\
0 & \text{ if } \cR(\vec y) < \theta_i
\end{cases} \, .
    \]
    Moreover, if $i < j$, we have $\Truth(f_{i}) \subset \Truth(f_{j})$, implying $f_{i} \prec f_{j}$.
  \end{proof}

\subsection{$\ast$-Realizability}\label{ast-Realizability}
In this section we discuss technical points needed later for  $\Sigma$, $\Pi\Sigma$, and $\Sigma\Pi\Sigma$ realizability.

Since $\Sigma$, $\Pi\Sigma$, and $\Sigma\Pi\Sigma$ realizability are based on $\Lambda$ functions, $\Lambda : \R^n_+ \to \R_+$, we need to consider a restricted class of realizing functions of the form 
\[ \cR := \Lambda \circ \phi, \]
 where $\phi : \B^n \to \R^n_+$ component-wise monotonically encodes a Boolean vector into a real valued vector, i.e.\
$\phi = (\phi_1,\dots,\phi_n)$, where $\phi_i:\B\to\R_+$, $\phi_i(0)<\phi_i(1)$ and $\Lambda$ is an algebraic function that belongs to one of the classes
$ \boldsymbol \Sigma{(n)} \subsetneq \boldsymbol \Pi\boldsymbol\Sigma{(n)}  \subsetneq \boldsymbol \Sigma \boldsymbol\Pi\boldsymbol\Sigma{(n)}$.

The following  Lemma is a direct consequence of Theorem~\ref{thm:K}, the definition of the classes of algebraic functions $ \boldsymbol \Sigma{(n)} $,  $\boldsymbol \Pi\boldsymbol\Sigma{(n)}  $ and $ \boldsymbol \Sigma \boldsymbol\Pi\boldsymbol\Sigma{(n)}$, and the  previously made observation that all switching systems are $K$-systems. 
 
\begin{lem}\label{lem:realizable}
In the following, $\ast$ could be $\Sigma$, $\Pi\Sigma$, or $\Sigma\Pi\Sigma$. A function $f \in $ \textbf{MBF}$^+$(n) is \textit{$\ast$-realizable} if, and only if, there exist a realizing network \textbf{RN} and realizing function $\cR = \Lambda \circ \phi$, where
\begin{enumerate}
    \item the $*$-interaction function $\Lambda: \R^n\to \R$ belongs to the class  $\Lambda \in  \boldsymbol \Sigma \boldsymbol\Pi\boldsymbol\Sigma{(n)}$, $\Lambda \in  \boldsymbol\Pi\boldsymbol\Sigma{(n)}$, or $\Lambda \in  \boldsymbol\Sigma{(n)}$  if $\ast = \Sigma\Pi\Sigma$, $\ast = \Pi\Sigma$, or $\ast = \Sigma$ respectively; and
    \item for each $i \in \{1,\dots,n\}$, the function $\phi_i : \B \to \R_+$ satisfies $\phi_i(0) < \phi_i(1)$.
\end{enumerate}

Similarly, a $k$-tuple of MBFs $f_1, f_2,\dots, f_k \in \MBF^+(n)$ is \emph{$\ast$-jointly realizable} if and only if there exist a realizing network \textbf{RN} and a realizing function $\cR=\Lambda \circ \phi$ for $\Lambda$ a $*$-interaction function from (1) and $\phi$ a map satisfying (2).
\end{lem}

The general question of which $k$-tuples of Boolean functions are $\ast$-jointly realizable seems very  difficult  and is likely connected to fundamental problems in algebraic geometry.
We focus here on some initial results for $k=2$ and will consider pairs of Boolean functions with different numbers of inputs.
We start with definitions and results that enumerate consequences of joint realizability of $f\prec g$ on relationships between $\Truth$ and $\False$ sets on subsets of the space of Boolean inputs. 

In the following, and many times throughout the rest of the manuscript, it will be useful to view $\B^n$ as a hypercube embedded in $\R^n_+$ with side lengths of 1. This gives rise to a geometrical structure of $\B^n$, where if $\vec y = (y_1,\dots,y_{i-1},0,y_{i+1},\dots,y_n)$, then $\vec y + \hat e_i = (y_1,\dots,y_{i-1},1,y_{i+1},\dots,y_n)$, where $\hat e_i$ is the standard $i$-th basis vector in $\R^n$. This defines the geometrical idea of \textit{floor} and \textit{ceiling} in the $i$-th direction of the hypercube $\B^n$. When $f \prec g$, there are relationships between the $\Truth$ and $\False$ sets on the floors and ceilings in all directions.

\begin{defn}
We define \textit{the ceiling (of $\B^n$) in the $i$-th normal direction} as the set
\begin{equation*}
\bC_i := \{(y_1,y_2,\dots,y_n) \in \B^n \mid y_i = 1\}
\end{equation*}
and similarly we define \textit{the floor (of $\B^n$) in the $i$-th normal direction} as the set
\begin{equation*}
\bF_i := \{(y_1,y_2,\dots,y_n) \in \B^n \mid y_i = 0\}.
\end{equation*}
Notice that $\bC_i$ and $\bF_i$ are both hypercubes of dimension $(n-1)$, that $\B^n = \bC_i \cup \bF_i$, and that $\bC_i = \bF_i + \hat e_i$.
\end{defn}

Next we define the idea of a \textit{collapse}, in which a floor and ceiling of $\B^n$ are considered objects embedded in the hypercube $\B^{n-1}$.

\begin{defn}
For a given $i \in \{1,\dots,n\}$, we define the \textit{$i$-th collapse} as the function $\ttCol_i: \B^n \to \B^{n-1}$ which removes the $i$-th coordinate, defined as
\begin{equation*}
\ttCol_i(\left(y_1,\dots,y_n\right)) := \left(y_1,\dots, y_{i-1}, y_{i+1},\dots,y_n\right) \end{equation*}
Then for any subset $U \subset \B^n$ we have
\begin{equation*}
\ttCol_i(U) = \{ \left(y_1,\dots, y_{i-1}, y_{i+1},\dots,y_n\right) \in \B^{n-1} \mid \left(y_1,\dots,y_n\right) \in U\} .
\end{equation*}
\end{defn}

Using the notions of floor, ceiling, and collapse, we move through a series of results that are critical to future proofs involving $\ast$-joint realizability for $\Lambda$ function classes.

\begin{lem}\label{lem:floorCeilingUpperset}
If $f\in \MBF^+(n)$, then for all $i \in \{1, \dots, n\}$, 
\begin{equation*}
\ttCol_i(\Truth(f)|_{\bF_i}) \subseteq \ttCol_i(\Truth(f)|_{\bC_i})
\end{equation*}
\end{lem}

\begin{proof}
Observe that the hypercube can be viewed as a distributive lattice via the relation $\leq$ on the corners of the hypercube by 
\[ \vec y \leq \vec z \in \B^n \iff \left(y_i =1 \Rightarrow z_i = 1\right) \text{ for all } i = 1,\dots,n.\] 
Notice that since any $f \in$ \textbf{MBF}$^+(n)$ is positive monotone, $\Truth(f)$ is an upperset of $\B^n$ viewed as a lattice.
Therefore, for any $\vec y \in \Truth(f)|_{\bF_i}$, we have $\vec y + \hat e_i \in \Truth(f)|_{\bC_i}$. Under the collapse operation, we have $\ttCol_i(\vec y) = \ttCol_i(\vec y + \hat e_i)$, completing the proof.
\end{proof}

\begin{prop}\label{prop:floorFloor}
If $f, g\in \MBF^+(n)$ and $f \prec g$, then for all $i \in \{1, \dots, n\}$, 
\begin{equation}\label{eq:fgfloorceil}
\ttCol_i(\Truth(f)|_{\bF_i}) \subseteq \ttCol_i(\Truth(g)|_{\bF_i}) \text{ and } \ttCol_i(\Truth(f)|_{\bC_i}) \subseteq \ttCol_i(\Truth(g)|_{\bC_i}).
\end{equation}
\end{prop}

\begin{proof}
Recall that $f \prec g$ implies $\Truth(f) \subseteq \Truth(g)$, which implies that $\Truth(f)|_{\bF_i} \subseteq \Truth(g)|_{\bF_i}$ and that $\Truth(f)|_{\bC_i} \subseteq \Truth(g)|_{\bC_i}$. Since in both cases, the collapse operation occurs over the same set, we have~\eqref{eq:fgfloorceil} as desired.
\end{proof}

\begin{prop}\label{prop:floorCeiling}
If $f, g\in \MBF^+(n)$ and $f \prec g$, then for all $i \in \{1, \dots, n\}$, 
\begin{equation*}
\ttCol_i(\Truth(f)|_{\bF_i}) \subseteq \ttCol_i(\Truth(g)|_{\bC_i})
\end{equation*}
\end{prop}

\begin{proof}
By Lemma~\ref{lem:floorCeilingUpperset} we see have $\ttCol_i(\Truth(f)|_{\bF_i}) \subseteq \ttCol_i(\Truth(f)|_{\bC_i})$, and by Proposition~\ref{prop:floorFloor} we have $\ttCol_i(\Truth(f)|_{\bC_i}) \subseteq \ttCol_i(\Truth(g)|_{\bC_i})$, which completes the proof.
\end{proof}


The following two technical results for special forms of $\Lambda$ functions are used extensively in the coming sections. Proofs for Theorem~\ref{thm:factorTermComparable} and Theorem~\ref{thm:extendedEmbedding} are found in Appendix~\ref{appendixProofs}.

\begin{defn}
Let $\Lambda$ be a $\ast$-interaction function. We say $z_i$ is a \textit{factor (of $\Lambda$)} if there is another map $\Lambda'$ that does not depend on $z_i$ such that $\Lambda = z_i \Lambda'$. Similarly we say $z_i$ is a \textit{simple term (of $\Lambda$)} if we can write $\Lambda$ as $\Lambda = z_i + \Lambda'$. Notice if $\Lambda$ is a $\Sigma$-interaction function, for all $i \in \{1, \dots, n\}$, $z_i$ is a simple term.
\end{defn}

\begin{thm}\label{thm:factorTermComparable}
Let $\ast$ be $\Sigma$, $\Pi\Sigma$, or $\Sigma\Pi\Sigma$.
Let $f, g\in \MBF^+(n)$, with $f \prec g$, be $\ast$-jointly realizable. Let $(\Lambda\circ\phi,$\textbf{RN}) $\ast$-jointly realize $(f,g)$. For each $\ell \in \{1, \dots, n\}$, if $z_\ell$ is a factor or a simple term of $\Lambda$, then
\begin{align}
\ttCol_\ell(\Truth(f)|_{\bC_\ell}) &\subseteq \ttCol_\ell(\Truth(g)|_{\bF_\ell}), \text{ or} \label{test}\\ 
\ttCol_\ell(\Truth(f)|_{\bC_\ell}) &\supseteq \ttCol_\ell(\Truth(g)|_{\bF_\ell}) \label{test2}
\end{align}
\end{thm}

\begin{thm}\label{thm:extendedEmbedding}
Let $\ast$ be $\Sigma$, $\Pi\Sigma$, or $\Sigma\Pi\Sigma$.
Let $n>1$. Let $f, g\in \MBF^+(n)$ be $\ast$-jointly realizable MBFs on $\B^n$ with $f \prec g$. Let $\ell \in \{1,\dots,n\}$ and $U \in \{\bF_\ell, \bC_\ell \}$. Then 
\begin{enumerate}
\item[a)] the functions $f'_U,g'_U \in \MBF^+(n-1)$ defined by
\[ \Truth(f'_U) = \ttCol_\ell(\Truth(f)\big|_{U}) \text{ and } \Truth(g'_U) = \ttCol_\ell(\Truth(g)\big|_{U})\]
are $\ast$-jointly realizable MBFs on $\B^{n-1}$, and 
\item[b)]  there is a single $\ast$-interaction function $\Lambda'$ along with maps $\phi_{\bC_\ell},\phi_{\bF_\ell}$ and weighted regulatory networks \textbf{RN}$_{\bC_\ell}$ and \textbf{RN}$_{\bF_\ell}$ such that $(\Lambda'\circ \phi_{\bF_\ell},$\textbf{RN}$_{\bF_\ell}$) $\ast$-jointly realizes $(f'_{\bF_\ell}, g'_{\bF_\ell})$
and $(\Lambda'\circ \phi_{\bC_\ell},$\textbf{RN}$_{\bC_\ell}$) $\ast$-jointly  realizes $(f'_{\bC_\ell}, g'_{\bC_\ell})$.
\end{enumerate}
\end{thm}

\subsection{Joint realizability in $\B^{n}$ and realizability in $\B^{n+1}$}

We will show that there is a bijection $\eta$ between pairs $(f,g) \in \MBF^+(n) \times \MBF^+(n)$ satisfying $f \prec g$ and $h \in \MBF^+(n+1)$. We use $\eta$  to  relate the $\ast$-joint realizability of a pair $f \prec g$ and the $\ast$-realizability of single function $\eta(f,g) \in \MBF^+(n+1)$. We will use this fact at the end of the section to prove rows 1 and 2 of Table~\ref{tableOfResults}.

\begin{defn}\label{defn:bijection}
Define the map
\[
\eta: \{(f,g) \mid f,g \in \MBF^+(n) \text{ and } f \prec g\} \to \MBF^+(n+1)
\]
by $h = \eta(f,g)$, where for $\vec y \in \B^{n+1}$,
\[
h(\vec y) = \begin{cases}
f(y_1,\dots,y_n) & \text{ if } y_{n+1}=0 \\
g(y_1,\dots,y_n) & \text{ if } y_{n+1}=1
\end{cases} \quad .
\]
\end{defn}

\begin{lem}
The map $\eta$ is a bijection onto $\MBF^+(n+1)$.
\end{lem}

\begin{proof} 
  First we describe the range of $\eta$. Let $f,g \in \MBF^+(n)$ with $f \prec g$. By definition $f$ and $g$ describe the floor and ceiling of $\eta(f,g) = h$ in the $(n+1)$-th direction, \begin{equation}\label{eq:truths}
  \Truth(f) = \ttCol_{n+1}\left(\Truth(h)\big|_{\bF_{n+1}}\right) \text{ and } \Truth(g) = \ttCol_{n+1}\left(\Truth(h)\big|_{\bC_{n+1}}\right). \end{equation} 
  Then the positive monotonicity of $f$ and $g$ induce positive monotonicity on the floor and ceiling of $h$, and $f\prec g$ gives positive monotonicity in the $(n+1)$-th direction. So the range of $\eta$ is contained in $\MBF^+(n+1)$.

It is clear that $\eta$ is injective.  Indeed, if $\eta(f,g) =\eta(f',g')$ then it follows immediately from the definition that $f=f'$ and $g=g'$. 

To show that $\eta$ is surjective, consider  $h \in \MBF^+(n+1)$ and define $f$ and $g$ by setting~\eqref{eq:truths} to be true. Since $h$ satisfies positive monotonicity on its floor and ceiling, $f,g \in \MBF^+(n)$. Also, since \[\ttCol_{n+1}\left(\Truth(h)\big|_{\bF_{n+1}}\right) \subseteq \ttCol_{n+1}\left(\Truth(h)\big|_{\bC_{n+1}}\right),\]  by  Lemma~\ref{lem:floorCeilingUpperset}, we have that $f \prec g$.  We have then constructed  the desired pair $(f,g)$ with $f \prec g$ such that $\eta(f,g)=h$ is well-defined.
\end{proof}

As a consequence of this result, note that  $\eta^{-1}(h) = (f,g)$  is well defined.

\begin{thm} \label{thm:fgImplies}
Let $n \geq 1$. Suppose $f \prec g$ is $\ast$-jointly realizable, where $f,g \in \MBF^+(n)$. Then $h = \eta (f,g)$ is $\ast$-realizable.
\end{thm}

\begin{proof}  Suppose $f \prec g$ is $\ast$-jointly realizable. Then by Lemma~\ref{lem:realizable} there exists $(\Lambda\circ\phi,$ \textbf{RN}) that $*$-jointly realizes $f$ and $g$. Moreover, the proof of Theorem~\ref{thm:KCompatible} tells us that the thresholds in \textbf{RN} associated to these functions satisfy $\theta_g < \theta_f$. We seek to construct $(\Lambda'\circ\phi',$ \textbf{RN}$')$ that $\ast$-realizes $h$. To build \textbf{RN}$'$, we take \textbf{RN} and add a source to the node under consideration from any other node in the network. It remains to discover the weight, $\theta'$ of that edge.

\textbf{Case 1: ($\ast = \Sigma$ or $\ast = \Sigma\Pi\Sigma$).} By the assumption since $f \prec g$ is $\ast$-jointly realizable, we have
\[
f(y_1,\dots,y_n) = \begin{cases}
1 &\text{if } \Lambda(\phi_1(y_1),\dots\phi_n(y_n)) > \theta_f \\
0 &\text{if } \Lambda(\phi_1(y_1),\dots\phi_n(y_n)) < \theta_f
\end{cases}
\]
\[
g(y_1,\dots,y_n) = \begin{cases}
1 &\text{if } \Lambda(\phi_1(y_1),\dots\phi_n(y_n)) > \theta_g \\
0 &\text{if } \Lambda(\phi_1(y_1),\dots\phi_n(y_n)) < \theta_g
\end{cases}
\]

For $\vec y \in \B^{n+1}$, we assign $\Lambda'\circ\phi' = \Lambda \circ\phi+ \phi'_{n+1}$ for $\phi' = (\phi_1,\dots,\phi_n,\phi'_{n+1})$ with some choice of $\phi'_{n+1}$. We know $f(y_1,\dots,y_n) = h(y_1,\dots,y_n,0)$ and $g(y_1,\dots,y_n) = h(y_1,\dots,y_n,1)$, so whatever $\phi'_{n+1}$ and $\theta'$ we choose must satisfy
\[
f(y_1,\dots,y_n) =h(y_1,\dots,y_n,0) =  \begin{cases}
1 &\text{if } \Lambda(\phi_1(y_1),\dots\phi_n(y_n)) +  \phi'_{n+1}(0)> \theta' \\
0 &\text{if } \Lambda(\phi_1(y_1),\dots\phi_n(y_n))+  \phi'_{n+1}(0) < \theta'
\end{cases}
\]
and 
\[
g(y_1,\dots,y_n) =h(y_1,\dots,y_n,1) =  \begin{cases}
1 &\text{if } \Lambda(\phi_1(y_1),\dots\phi_n(y_n)) +  \phi'_{n+1}(1)> \theta' \\
0 &\text{if } \Lambda(\phi_1(y_1),\dots\phi_n(y_n))+  \phi'_{n+1}(1) < \theta'
\end{cases}
\]
Consider the assignment $\phi'_{n+1}(0)=\epsilon$, $\phi'_{n+1}(1)=\theta_f + \epsilon -\theta_g$, and $\theta' := \theta_f+\epsilon$, where $\epsilon$ is any sufficiently small real number $0<\epsilon < \theta_f-\theta_g$. 
It is easy to check that with this assignment, $h(y_1,\dots,y_n,0) = 1$ if and only if $f(y_1,\dots,y_n)=1$ and that $h(y_1,\dots,y_n,1) = 1$ if and only if $g(y_1,\dots,y_n)=1$.
This completes the construction of  $(\Lambda'\circ \phi',$ \textbf{RN}$')$ that $\ast$-realizes $h$.

\textbf{Case 2: ($\ast = \Pi\Sigma$)} The proof proceeds analogously with Case 1, where the only difference is replacement of a simple term by a factor in $\Lambda'$. It is easy to verify that the  following assignments $\ast$-realize $h$: $\phi'=(\phi_1,\dots,\phi_n,\phi'_{n+1} )$, 
\[
\Lambda'\circ\phi'(y_1,\dots,y_n,y_{n+1}) = \phi'_{n+1}(y_{n+1}) \cdot \left(\Lambda\circ\phi(y_1,\dots,y_n) \right)\ ,
\] $\theta' = \theta_f$, $\phi'_{n+1}(0)=1$, $\phi_{n+1}(1) = \theta_f / \theta_g$.
\end{proof}

We do not know if the converse of Theorem~\ref{thm:fgImplies} is true in general.
 However, with an additional constraint we obtain the following theorem.

\begin{thm} \label{thm:hImplies}
Let $n \geq 1$. Suppose $(\Lambda\circ\phi,$\textbf{RN}) \textit{$\ast$-realizes} $h \in \MBF(n+1)$. If $z_i$ is a factor or simple term of $\Lambda$, then $f,g \in \MBF^+(n)$ defined by 
\[ \Truth(f) = \ttCol_{i}(\Truth(h)|_{\bF_i}),\quad \Truth(g) = \ttCol_{i}(\Truth(h)|_{\bC_i})\] are $\ast$-jointly realizable.
\end{thm}

\begin{proof}
Without loss of generality assume $i = n+1$. Let $\theta$ be the threshold associated to the realization of $h$.

\textbf{Case 1: ($z_{n+1}$ is a factor)} Then   $\Lambda =  z_{n+1} \Lambda'$ and from Lemma~\ref{lem:realizable}
\[
h(y_1,\dots,y_{n+1}) = \begin{cases}
1 &\text{if } \phi_{n+1}(y_{n+1})\Lambda'(\phi_1(y_1),\dots\phi_n(y_n)) > \theta \\
0 &\text{if } \phi_{n+1}(y_{n+1})\Lambda'(\phi_1(y_1),\dots\phi_n(y_n)) < \theta
\end{cases}
\]

If we restrict our attention to $f$, dividing by $\phi_{n+1}(0)$, the above equation gives us
\[
f(y_1,\dots,y_n) = h(y_1,\dots,y_n,0) = \begin{cases}
1 &\text{if } \Lambda'(\phi_1(y_1),\dots\phi_n(y_n)) > \theta/\phi_{n+1}(0) \\
0 &\text{if } \Lambda'(\phi_1(y_1),\dots\phi_n(y_n)) < \theta/\phi_{n+1}(0)
\end{cases} \quad .
\]
Restricting our attention to $g$ we see that
\[
g(x_1,\dots,x_n) = h(x_1,\dots,x_n,1) = \begin{cases}
1 &\text{if } \Lambda'(\phi_1(x_1),\dots\phi_n(x_n)) > \theta/\phi_{n+1}(1) \\
0 &\text{if } \Lambda'(\phi_1(x_1),\dots\phi_n(x_n)) < \theta/\phi_{n+1}(1). 
\end{cases}
\]
Construct \textbf{RN}$'$ by removing the source edge associated to $n+1$ and adding one target edge to the node under consideration. Assign to one target edge the weight $\theta_f = \theta/\phi_{n+1}(0)$ and assign $\theta_g=\theta/\phi_{n+1}(1)$ to the other. After setting $\phi' = (\phi_1,\dots,\phi_n)$, we obtain  $(\Lambda'\circ\phi',$ \textbf{RN}$')$ that $\ast$-jointly realizes $(f,g)$.

\textbf{Case 2: ($z_{n+1}$ is a simple term)} The argument for this case is similar but  instead of dividing by $\phi_{n+1}(y_{n+1})$, we will  subtract. Specifically, since $y_{n+1}$ is a simple term,
\[
h(y_1,\dots,y_{n+1}) = \begin{cases}
1 &\text{if } \phi_{n+1}(y_{n+1}) + \Lambda'(\phi_1(y_1),\dots\phi_n(y_n)) > \theta \\
0 &\text{if } \phi_{n+1}(y_{n+1}) + \Lambda'(\phi_1(y_1),\dots\phi_n(y_n)) < \theta
\end{cases}
\]
and so 
\[
f(y_1,\dots,y_n) = h(y_1,\dots,y_n,0) = \begin{cases}
1 &\text{if } \Lambda'(\phi_1(y_1),\dots\phi_n(y_n)) > \theta - \phi_{n+1}(0) \\
0 &\text{if } \Lambda'(\phi_1(y_1),\dots\phi_n(y_n)) < \theta - \phi_{n+1}(0)
\end{cases}
\]
\[
g(y_1,\dots,y_n) = h(y_1,\dots,y_n,1) = \begin{cases}
1 &\text{if } \Lambda'(\phi_1(y_1),\dots\phi_n(y_n)) > \theta - \phi_{n+1}(1) \\
0 &\text{if } \Lambda'(\phi_1(y_1),\dots\phi_n(y_n)) < \theta -\phi_{n+1}(1)
\end{cases}
\]
 Construct \textbf{RN}$'$ as before with threshold assignments $\theta_f = \max\{0,\theta - \phi_{n+1}(0)\}$, $\theta_g =\max\{0,\theta - \phi_{n+1}(1)\}$, and a further perturbation by small enough $\ep>0$ if $\theta_f=\theta_g$. Then tuple $(\Lambda'\circ\phi',$ \textbf{RN}$')$ $\ast$-jointly realizes $(f,g)$.
\end{proof}

The following Corollary is an immediate result of Theorems~\ref{thm:fgImplies} and~\ref{thm:hImplies}. It states that, in the $ \Sigma$ class of functions, joint realizability of a pair $(f,g)$ in dimension $n$ is equivalent to the realizability of $\eta(f,g)$ in dimension $n+1$, since every term in $\Lambda$ is simple.

\begin{cor}\label{proofOfRows1and2}
Let $n \geq 1$. Suppose $f \prec g$ and let $h = \eta (f,g)$. Then 
$(f,g)$ is $\Sigma$-jointly realizable if and only if $h$ is $\Sigma$-realizable.
\end{cor}

As promised, we now show the equivalence of threshold (linearly separable) functions and $\Sigma$-realizability, see Definition~\ref{def:threshold}.

\begin{lem}\label{lem:SigmaIsThreshold}
Let $f \in \MBF^+(n)$.
\begin{enumerate}
\item If $f$ is $\Sigma$-realizable then $f$ is a threshold function.

\item If $f$ is a threshold function with separating structure $(a_1,\dots,a_n,\theta')$ such that $a_1,\dots,a_n\geq 0$ and $\theta' > -n$, then $f$ is $\Sigma$-realizable.
\end{enumerate}
\end{lem}

\begin{proof}
(1) Suppose $f$ is $\Sigma$-realizable, and let $(\Lambda\circ\phi,$ \textbf{RN}$)$ $\Sigma$-realize $f$. We construct $a_1, \dots, a_n$ and $\theta'$ as in the sense of Definition~\ref{def:threshold} as follows: set $a_i = \phi_i(1) - \phi_i(0)$, and let $\theta' = \max\{0,\theta - (\phi_1(0) + \dots + \phi_n(0))\}$.

(2) Now suppose $f$ is a threshold function with separating structure $(a_1,\dots,a_n,\theta')$ such that $a_1,\dots,a_n \geq 0$ and $\theta' > -n$. Set $\Lambda = z_1 + \dots + z_n$. Set $\phi_i(0)=1$, $\phi_i(1) = 1 + a_i$, and $\theta = \theta' + n$, to obtain the desired $(\Lambda\circ\phi, \theta)$.
\end{proof}

The following proposition, together with Remark~\ref{rem:realizabilitySpecialCases}, proves the first two rows of Table~\ref{tableOfResults}.
\begin{prop}\label{first:two}
Assume $f\prec g \in \MBF^+(n)$ with $n=1$ or $n=2$. Then the pair $(f,g)$ is $\Sigma$-realizable. 
\end{prop}
\begin{proof}
By simple enumeration, one can check that, for $n = 1,2,3$, all functions in $\MBF^+(n)$ are threshold functions, and admit separating structures with $a_1,\dots,a_n,\theta >0$. Via Lemma~\ref{lem:SigmaIsThreshold}, these functions are $\Sigma$-realizable. This fact, when combined with Corollary~\ref{proofOfRows1and2}, proves the first two rows of Table~\ref{tableOfResults}. 
 \end{proof}



\subsection{Strict subset relations in Table~\ref{tableOfResults}.} This section contains a series of examples illustrating the differences between $\Sigma$, $\Pi\Sigma$, and $\Sigma\Pi\Sigma$ realizability, proving some of the strict subset results in Table~\ref{tableOfResults}. We will use Theorems~\ref{thm:factorTermComparable} and~\ref{thm:extendedEmbedding} extensively. 

The idea behind all of the examples is to show that there exists a pair of $\ast$-jointly realizable functions $(f,g)$ that are not $\ast'$-jointly realizable, where $\ast'$ is a more restrictive class than $\ast$. The proofs are inductive, with different base case constructions and very similar inductive steps. The methodology for the induction is to take an $(f,g)$ $\ast$-jointly realizable, but not $\ast'$-jointly realizable, pair in $\B^n$ and to set $(f,g)$ to be the floors of new $(\tilde f, \tilde g)$ functions in $\B^{n+1}$. It then remains to construct ceilings that ensure $\tilde f, \tilde g \in \MBF^+(n+1)$. For any $\vec y \in \bC_{n+1}$, we choose to set $\tilde f(\vec y) = \tilde g(\vec y) = 1$. By this choice, 
\[ \Truth(\tilde f) \supseteq \bC_{n+1}, \quad \Truth(\tilde g) \supseteq \bC_{n+1},\] which ensures that $\tilde f, \tilde g \in \MBF^+(n+1)$. By the contrapositive of Theorem~\ref{thm:extendedEmbedding} (a), this construction ensures that $(\tilde f, \tilde g)$ are not $\ast'$-jointly realizable. We then show that $(\tilde f, \tilde g)$ are $\ast$-jointly realizable.

\subsubsection{$\Sigma$-jointly realizable $\subsetneq$ $\Pi\Sigma$-jointly realizable for $n\geq 3$.}\label{subsection:SigmaSubsetPiSigma}

\begin{figure}[h]
\begin{minipage}{.49\textwidth}
\centering
\begin{tikzpicture}[rotate around x=0, rotate around y=0, rotate around z=0, main node/.style={circle, draw, thick, inner sep=1pt, minimum size=0pt}, scale=2.8, node distance=1cm,z={(60:-0.5cm)}]
\tikzstyle{every loop}=[looseness=14]

\node[main node, fill=darkgray] (000) at (0,0,0) {$000$} ;
\node[main node, fill=darkgray] (001) at (0,0,1) {$001$} ;
\node[main node, fill=lightgray] (010) at (0,1,0) {$010$} ;
\node[main node, fill=lightgray] (100) at (1,0,0) {$100$} ;
\node[main node, fill=lightgray] (110) at (1,1,0) {$110$} ;
\node[main node, fill=white] (101) at (1,0,1) {$101$} ;
\node[main node, fill=white] (011) at (0,1,1) {$011$} ;
\node[main node, fill=white] (111) at (1,1,1) {$111$} ;

\draw[very thick,-,shorten >= 0pt,shorten <= 0pt]
(001) edge[] (101)
(001) edge[] (000)
(000) edge[] (100)
(101) edge[] (100)
(011) edge[] (010)
(010) edge[] (110)
(001) edge[] (011)
(000) edge[] (010)
(100) edge[] (110)
(111) edge[] (110)

(011) edge[line width=0.1cm, white] (111)
(011) edge[] (111)
(101) edge[line width=0.1cm, white] (111)
(101) edge[] (111);
\end{tikzpicture}
\end{minipage}
\hfill
\begin{minipage}{.49\textwidth}
\centering
\begin{tabular}{|c|c|c|} \hline
$y_1y_2y_3$ & $\phi(y_1y_2y_3)$ & $\Lambda(\phi(y_1y_2y_3))$  \\ \hline \rowcolor{darkgray}
 000 & (1,1,1) &  2 \\ \rowcolor{darkgray}
 001 & (1,1,2) &  4 \\ \rowcolor{lightgray}
 100 & (4,1,1) &  5 \\\rowcolor{lightgray}
 010 & (1,4,1) &  5 \\\rowcolor{lightgray}
 110 & (4,4,1) &  8 \\
 101 & (4,1,2) & 10 \\
 011 & (1,4,2) & 10 \\
 111 & (4,4,2) & 16 \\ \hline
\end{tabular}
\end{minipage}
\caption{Left: An example pair $f, g: \B^3 \to \B$ with $f \prec g$. Dark grey is $\False(f) \cap \False(g)$, light grey is $\False(f) \cap \Truth(g)$, and white is $\Truth(f) \cap \Truth(g)$. Nodes are labels with $y_1y_2y_3$. The pair $(f,g)$ are $\Pi\Sigma$-jointly realizable, but not $\Sigma$-jointly realizable. Right: A table of values proving $(f,g)$ are $\Pi\Sigma$-jointly realizable. Here $\Lambda = (z_1+z_2)z_3$ and $\theta_1=4.5$, $\theta_2=9$. The coloring in the table column is consistent with vertex coloring on the left.} \label{fig:linearSubsetDSGRN}
\end{figure}
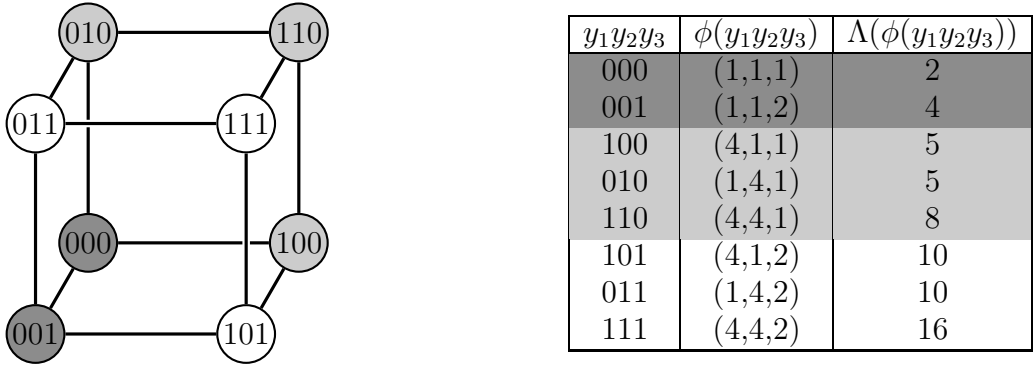

In this section we prove the first strict inclusion in the third and fourth rows of  Table~\ref{tableOfResults}.

\begin{lem}\label{lem:sigmavspisigma}
Let $n\geq 3$. There exists a pair $f\prec g \in \MBF^+(n)$ such that $(f,g)$ is not $\Sigma$-jointly realizable, but is $\Pi\Sigma$-jointly realizable.
\end{lem}

\begin{proof}
We first construct an explicit pair $(f,g)$ for $n=3$. Consider the pair $f\prec g$ of MBFs depicted on the left of Figure~\ref{fig:linearSubsetDSGRN}. We use Theorem~\ref{thm:factorTermComparable} to show that $(f,g)$ is not $\Sigma$-jointly realizable, and provide an explicit $(\Lambda\circ\phi,$ \textbf{RN}) that $\Pi\Sigma$-jointly realizes $(f,g)$. Choose any \textbf{RN} with a node with three sources and two targets, with threshold values to be determined.

First, we illustrate the use of Theorem~\ref{thm:factorTermComparable}. Observe that $\Truth(f)|_{\bC_1} = \{101,111\}$, and so $\ttCol_1(\Truth(f)|_{\bC_1})= \{01,11\}$. Similarly $\Truth(g)|_{\bF_1} = \{010,011\}$, so $\ttCol_1(\Truth(g)|_{\bF_1}) = \{10,11\}$. Therefore, we can see that
\[
\ttCol_1(\Truth(f)|_{\bC_1}) \nsubseteq \ttCol_1(\Truth(g)|_{\bF_1}) \text{ and } \ttCol_1(\Truth(f)|_{\bC_1}) \nsupseteq \ttCol_1(\Truth(g)|_{\bF_1}) \ .
\]
By the contrapositive of Theorem~\ref{thm:factorTermComparable}, we see that if $(\Lambda\circ\phi,$ \textbf{RN}) $\ast$-jointly realizes $(f,g)$, then $\Lambda$ cannot have a simple term or factor $z_1$. Therefore, $(f,g)$ is not $\Sigma$-jointly realizable, as any $\Sigma$-interaction function $\Lambda$ has every variable as a simple term.

However, $(f,g)$ is $\Pi\Sigma$-jointly realizable. To see this, set $\Lambda = (z_1+z_2)z_3$, $\phi_1(0)=\phi_2(0)=\phi_3(0)=1$, $\phi_1(1)=4$, $\phi_2(1)=4.1$, $\phi_3(1)=2$, and $\theta_1=4.5$, $\theta_2=9$. The results of such an assignment are displayed in the table in Figure~\ref{fig:linearSubsetDSGRN}.

We now prove the inductive step. Let $n\geq 3$. Assume there exists $f,g : \B^n \to \B$ such that $(f,g)$ is not $\Sigma$-jointly realizable, but is $\Pi\Sigma$-jointly realizable. Let$(\Lambda\circ\phi,$ \textbf{RN}) with thresholds $\theta_1$ and $\theta_2$ $\Pi\Sigma$-jointly realize $(f,g)$. 
Now over $\B^{n+1}$ define 
\[
\tilde f(y_1\dots y_{n+1}) := \begin{cases}
f(y_1\dots y_n) & y_{n+1}=0 \\
1 & y_{n+1}=1
\end{cases}
\]

\[
\tilde g(y_1\dots y_{n+1}) := \begin{cases}
g(y_1\dots y_n) & y_{n+1}=0 \\
1 & y_{n+1}=1
\end{cases}
\]
Observe that 

\[
  \ttCol_{n+1}\left(\Truth(\tilde f)|_{\bF_{n+1}}\right)=\Truth(f), \quad \ttCol_{n+1}\left(\Truth(\tilde g)|_{\bF_{n+1}}\right)=\Truth(g),
  \]
  in other words the floor of $\tilde f$ has the same truth set as $f$ and the floor of $\tilde g$ has the same truth set as $g$. Since $(f,g)$ are not $\Sigma$-jointly realizable, the contrapositive of Theorem~\ref{thm:extendedEmbedding} (a) tells us that $(\tilde f, \tilde g)$ are not $\Sigma$-jointly realizable. Let 
$m = \min\{\Lambda(\phi(\B^n))\}$ and define $\tilde \phi_{n+1} (0) =1$ and $\tilde \phi_{n+1} (1)=\max\{2,C\}$, where $C$ is large enough such that $mC > \max\{\theta_1,\theta_2\}$.
Define $\tilde \Lambda = z_{n+1} \Lambda$ and $\tilde \phi:=(\phi_1 , \dots, \phi_n, \tilde \phi_{n+1})$. Then 
$\tilde \Lambda$ is a valid $\Pi\Sigma$-interaction function, and $(\tilde\Lambda\circ\tilde\phi,\mbox{\textbf{RN}})$ $\Pi\Sigma$-jointly realizes $(\tilde f, \tilde g)$, completing the proof.
\end{proof}
\subsubsection{$\Pi\Sigma$-jointly realizable $\subsetneq$ $\Sigma\Pi\Sigma$-jointly realizable for $n\geq 3$}\label{subsection:PiSigmaSubsetSigmaPiSigma}

In this section we prove the second strict inclusion in the third and fourth  rows of  Table~\ref{tableOfResults}.

\begin{figure}[h]
  \begin{minipage}{.49\textwidth}
  \centering
  \begin{tikzpicture}[rotate around x=0, rotate around y=0, rotate around z=0, main node/.style={circle, draw, thick, inner sep=1pt, minimum size=0pt}, scale=2.8, node distance=1cm,z={(60:-0.5cm)}]
  \tikzstyle{every loop}=[looseness=14]
  
  \node[main node, fill=darkgray] (000) at (0,0,0) {$000$} ;
  \node[main node, fill=lightgray] (001) at (0,0,1) {$001$} ;
  \node[main node, fill=darkgray] (010) at (0,1,0) {$010$} ;
  \node[main node, fill=darkgray] (100) at (1,0,0) {$100$} ;
  \node[main node, fill=white] (110) at (1,1,0) {$110$} ;
  \node[main node, fill=lightgray] (101) at (1,0,1) {$101$} ;
  \node[main node, fill=lightgray] (011) at (0,1,1) {$011$} ;
  \node[main node, fill=white] (111) at (1,1,1) {$111$} ;
  
  \draw[very thick,-,shorten >= 0pt,shorten <= 0pt]
  (001) edge[] (101)
  (001) edge[] (000)
  (000) edge[] (100)
  (101) edge[] (100)
  (011) edge[] (010)
  (010) edge[] (110)
  (001) edge[] (011)
  (000) edge[] (010)
  (100) edge[] (110)
  (111) edge[] (110)
  
  (011) edge[line width=0.1cm, white] (111)
  (011) edge[] (111)
  (101) edge[line width=0.1cm, white] (111)
  (101) edge[] (111);
  \end{tikzpicture}
  \end{minipage}
  \hfill
  \begin{minipage}{.49\textwidth}
  \centering
  \begin{tabular}{|c|c|c|} \hline
  $y_1y_2y_3$ & $\phi(y_1y_2y_3)$ & $\Lambda(\phi(y_1y_2y_3))$  \\ \hline \rowcolor{darkgray}
   000 & (1,1,1) & 2 \\ \rowcolor{darkgray}
   100 & (3,1,1) & 4 \\ \rowcolor{darkgray}
   010 & (1,3.1,1) & 4.1 \\ \rowcolor{lightgray}
   001 & (1,1,4) & 5 \\\rowcolor{lightgray}
   101 & (3,1,4) & 7 \\\rowcolor{lightgray}
   011 & (1,3.1,4) & 7.1 \\ 
   110 & (3,3.1,1) &  10.3 \\
   111 & (3,3.1,4) &  13.3 \\ \hline
  \end{tabular}
  \end{minipage}
  \caption{An example pair $f, g: \B^3 \to \B$ with $f \prec g$. Dark grey is $\False(f) \cap \False(g)$, light grey is $\False(f) \cap \Truth(g)$, and white is $\Truth(f) \cap \Truth(g)$. Nodes are labels with $y_1y_2y_3$. The pair $(f,g)$ is $\Sigma\Pi\Sigma$-jointly realizable, but not $\Pi\Sigma$-jointly realizable. Right: A table of values proving $(f,g)$ are $\Sigma\Pi\Sigma$-jointly realizable. Here $\Lambda = z_1z_2 + z_3$ and and $\theta_1=4.5$, $\theta_2=9$. The coloring in the rightmost column is consistent with vertex coloring on the left.} \label{fig:DSGRNSubseteDSGRN}
  \end{figure}
  
\begin{lem}\label{lem:DSGRNSubseteDSGRN}
Let $n\geq 3$. There exists a pair $f\prec g \in \MBF^+(n)$ such that $(f,g)$ is not $\Pi\Sigma$-jointly realizable, but is $\Sigma\Pi\Sigma$-jointly realizable.
\end{lem}

\begin{proof}
Again, we construct an explicit pair $(f,g)$ for $n=3$. Consider the pair $f \prec g$ shown in Figure~\ref{fig:DSGRNSubseteDSGRN}. We will show that $(f,g)$ is not $\Pi\Sigma$-jointly realizable. Suppose, by way of contradiction, that $(\Lambda\circ\phi,$ \textbf{RN}) $\Pi\Sigma$-jointly realizes $(f,g)$ for some \textbf{RN} with a node with three sources and two targets.

From Lemma~\ref{lem:realizable}, we can see that the only allowable $\Pi\Sigma$-interaction function for $n=3$ are
\[
z_1 + z_2 + z_3, \quad
(z_1 + z_2) z_3, \quad
(z_1 + z_3) z_2, \quad
(z_2 + z_3) z_1, \quad
z_1 z_2 z_3
\]
For $f \prec g$ in Figure~\ref{fig:DSGRNSubseteDSGRN}, observe that 
\begin{equation}\label{eq:problemDirectCube7_1}
\ttCol_1(\Truth(f)|_{\bC_1}) \nsubseteq \ttCol_1(\Truth(g)|_{\bF_1}) \text{ and } \ttCol_1(\Truth(f)|_{\bC_1}) \nsupseteq \ttCol_1(\Truth(g)|_{\bF_1}),
\end{equation}
and
\begin{equation}\label{eq:problemDirectCube7_2}
\ttCol_2(\Truth(f)|_{\bC_2}) \nsubseteq \ttCol_2(\Truth(g)|_{\bF_2}) \text{ and } \ttCol_2(\Truth(f)|_{\bC_2}) \nsupseteq \ttCol_2(\Truth(g)|_{\bF_2}).
\end{equation}
By Theorem~\ref{thm:factorTermComparable}, we see that $z_1$ and $z_2$ cannot be simple terms or factors of $\Lambda$. This constraint implies that $\Lambda = (z_1 + z_2) z_3$. The following inequality argument will show that this choice of $\Lambda$ is also impossible. To reduce notation, we will write $\phi_i(0)=\ell_i$ and $\phi_i(1)=u_i$.

We have the following relations from Figure~\ref{fig:DSGRNSubseteDSGRN}:
\begin{align*}
\Lambda(\phi(100)) & < \Lambda(\phi(001)) \quad (\text{dark grey} < \text{light grey})\\
\Lambda(\phi(010)) & < \Lambda(\phi(001)) \quad (\text{dark grey} < \text{light grey})\\
-&-\\
\Lambda(\phi(101)) & < \Lambda(\phi(110)) \quad (\text{light grey} < \text{white})\\
\Lambda(\phi(011)) & < \Lambda(\phi(110)) \quad (\text{light grey} < \text{white})
\end{align*}

Written in the language of $\ell$ and $u$ this means
\begin{align*}
(u_1+\ell_2)\ell_3 &< (\ell_1+\ell_2)u_3  \\
(\ell_1+u_2)\ell_3 &< (\ell_1+\ell_2)u_3  \\
(u_1+\ell_2)u_3 &< (u_1+u_2)\ell_3  \\
(\ell_1+u_2)u_3 &< (u_1+u_2)\ell_3  .
\end{align*}
We consider first and fourth  equation; the second and third together  lead to similar contradiction.
First equation:
\begin{align*}
(u_1+\ell_2)\ell_3 &< (\ell_1+\ell_2)u_3  \\
u_1\ell_3 +\ell_2\ell_3 &< \ell_1u_3 +\ell_2u_3  \\
u_1\ell_3 - \ell_1u_3 &< \ell_2(u_3 -\ell_3)  
\end{align*}
Fourth equation:
\begin{align*}
(\ell_1+u_2)u_3 &< (u_1+u_2)\ell_3  \\
\ell_1u_3 +u_2u_3 &< u_1\ell_3 +u_2\ell_3 \\
u_2(u_3-\ell_3)  &< u_1\ell_3 - \ell_1u_3
\end{align*}
Comparing the last line in each equation block we get 
\[ u_2(u_3-\ell_3) <  u_1\ell_3 - \ell_1u_3 <   \ell_2(u_3 -\ell_3), \]
which, after cancellation, gives 
\[ u_2 < \ell_2.\]
Therefore $\phi_2(0) > \phi_2(1)$, contradicting Lemma~\ref{lem:realizable}, so $\Lambda= (z_1+z_2)z_3$ is also impossible. Therefore, $(f,g)$ is not $\Pi\Sigma$-jointly realizable.

However, the pair $(f,g)$ is $\Sigma\Pi\Sigma$-jointly realizable. To see this, let $\Lambda = z_1z_2+z_3$, let $\phi_1(0) = \phi_2(0) = \phi_3(0) = 1$, $\phi_1(1) =3$,  $\phi_2(1) = 3.1$, and $\phi_3(1) = 4$, and let $\theta_1=4.5$, and $\theta_2=9$. Such an assignment is displayed in the table in Figure~\ref{fig:DSGRNSubseteDSGRN}.

We now prove the inductive step. Let $n\geq 3$. Assume there exists $f,g : \B^n \to \B$ such that $(f,g)$ is not $\Pi\Sigma$-jointly realizable, but is $\Sigma\Pi\Sigma$-jointly realizable. Let $(\Lambda\circ\phi,$ \textbf{RN}) with thresholds $\theta_1$ and $\theta_2$ $\Sigma\Pi\Sigma$-jointly realize $(f,g)$.
Define 
\[
\tilde f(y_1\dots y_{n+1}) := \begin{cases}
f(y_1\dots y_n) & y_{n+1}=0 \\
1 & y_{n+1}=1
\end{cases}
\]
\[
\tilde g(y_1\dots y_{n+1}) := \begin{cases}
g(y_1\dots y_n) & y_{n+1}=0 \\
1 & y_{n+1}=1
\end{cases}
\]
As in the proof of Lemma~\ref{lem:sigmavspisigma}, observe that 
\[
  \ttCol_{n+1}\left(\Truth(\tilde f)|_{\bF_{n+1}}\right)=\Truth(f), \quad \ttCol_{n+1}\left(\Truth(\tilde g)|_{\bF_{n+1}}\right)=\Truth(g),
  \]
  in other words the floor of $\tilde f$ has the same truth set as $f$ and the floor of $\tilde g$ has the same truth set as $g$. Since $(f,g)$ are not $\Pi\Sigma$-jointly realizable, the contrapositive of Theorem~\ref{thm:extendedEmbedding} (a) tells us that $(\tilde f, \tilde g)$ are not $\Pi\Sigma$-jointly realizable.
Let 
$m = \min\{\Lambda(\phi(\B^n))\}$. Define $\tilde \phi_{n+1} (0) :=1$ and $\tilde \phi_{n+1} (0) :=\max\{2,C\}$, where $C$ is large enough such that $m + C > \max\{\theta_1,\theta_2\}$.
Define $\tilde \Lambda := z_{n+1} + \Lambda$ and $\tilde \phi:=(\phi_1 , \dots, \phi_n, \tilde \phi_{n+1})$. Then 
$\tilde \Lambda$ is a valid $\Sigma\Pi\Sigma$-interaction function, and $(\tilde\Lambda\circ\tilde\phi,$ \textbf{RN})  $\Sigma\Pi\Sigma$-jointly realizes $(\tilde f, \tilde g)$, completing the proof.
\end{proof}

\subsubsection{$\Sigma\Pi\Sigma$-jointly realizable $\subsetneq$ $K$-jointly realizable for $n \geq 4$}\label{subsection:SigmaPiSigmaSubsetK}
In this section we prove the last strict inclusion in the fourth row of  Table~\ref{tableOfResults}.

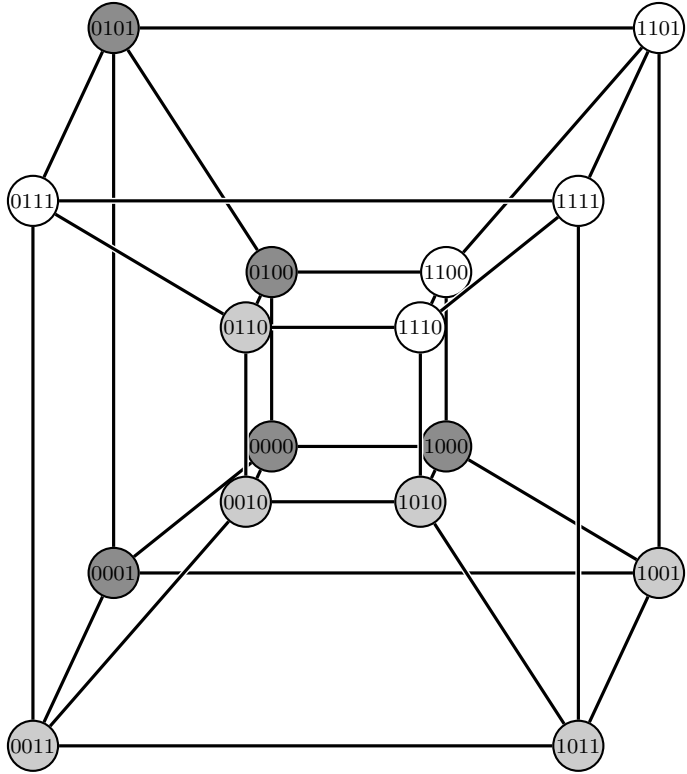
\begin{figure}
  \centering
  \begin{tikzpicture}[rotate around x=0, rotate around y=0, rotate around z=0, main node/.style={circle, draw, thick, inner sep=.5pt, minimum size=0pt}, scale=2.9, node distance=1cm,z={(65:-0.35cm)}]
  \tikzstyle{every loop}=[looseness=14]
  
  \node[main node, fill=darkgray] (0000) at (-.4,-.4,-.4) {\tiny $0000$} ;
  \node[main node, fill=darkgray] (0100) at (-.4,.4,-.4) {\tiny $0100$} ;
  \node[main node, fill=darkgray] (1000) at (.4,-.4,-.4) {\tiny $1000$} ;
  \node[main node, fill=white] (1100) at (.4,.4,-.4) {\tiny $1100$} ;
  
  \begin{pgfonlayer}{fg}
  \node[main node, fill=lightgray] (0010) at (-.4,-.4,.4) {\tiny $0010$} ;
  \node[main node, fill=lightgray] (1010) at (.4,-.4,.4) {\tiny $1010$} ;
  \node[main node, fill=lightgray] (0110) at (-.4,.4,.4) {\tiny $0110$} ;
  \node[main node, fill=white] (1110) at (.4,.4,.4) {\tiny $1110$} ;
  \end{pgfonlayer}
  
  \node[main node, fill=darkgray] (0001) at (-1.25,-1.25,-1.25) {\tiny $0001$} ;
  \node[main node, fill=darkgray] (0101) at (-1.25,1.25,-1.25) {\tiny $0101$} ;
  \node[main node, fill=lightgray] (1001) at (1.25,-1.25,-1.25) {\tiny $1001$} ;
  \node[main node, fill=white] (1101) at (1.25,1.25,-1.25) {\tiny  $1101$} ;
  
  \begin{pgfonlayer}{fg}
  \node[main node, fill=lightgray] (0011) at (-1.25,-1.25,1.25) {\tiny $0011$} ;
  \node[main node, fill=lightgray] (1011) at (1.25,-1.25,1.25) {\tiny $1011$} ;
  \node[main node, fill=white] (0111) at (-1.25,1.25,1.25) {\tiny $0111$} ;
  \node[main node, fill=white] (1111) at (1.25,1.25,1.25) {\tiny $1111$} ;
  \end{pgfonlayer}
  
  \draw[very thick,-,shorten >= 0pt,shorten <= 0pt]
  (0010) edge[] (1010)
  (0010) edge[] (0000)
  (0000) edge[] (1000)
  (1010) edge[] (1000)
  (0110) edge[] (0100)
  (0110) edge[] (1110)
  (0100) edge[] (1100)
  (1110) edge[] (1100)
  
  (0000) edge[] (0100)
  (1000) edge[] (1100)
  
  (0011) edge[] (0001)
  (0001) edge[] (1001)
  (1011) edge[] (1001)
  
  (0111) edge[] (0101)
  (0101) edge[] (1101)
  (1111) edge[] (1101)
  (0011) edge[] (0111)
  (0001) edge[] (0101)
  (1001) edge[] (1101)
  
  (0000) edge[] (0001)
  (1000) edge[] (1001)
  (0100) edge[] (0101)
  (1100) edge[] (1101)
  
  (0010) edge[line width=0.08cm, white] (0110)
  (0010) edge[] (0110)
  
  (1010) edge[line width=0.08cm, white] (1110)
  (1010) edge[] (1110)
  
  (0110) edge[line width=0.08cm, white] (1110)
  (0110) edge[] (1110)
  
  (0011) edge[line width=0.08cm, white] (1011)
  (0111) edge[line width=0.08cm, white] (1111)
  (1011) edge[line width=0.08cm, white] (1111)
  (0010) edge[line width=0.08cm, white] (0011)
  (1010) edge[line width=0.08cm, white] (1011)
  (0110) edge[line width=0.08cm, white] (0111)
  (1110) edge[line width=0.08cm, white] (1111)
  
  (0011) edge[] (1011)
  (0111) edge[] (1111)
  (1011) edge[] (1111)
  (0010) edge[] (0011)
  (1010) edge[] (1011)
  (0110) edge[] (0111)
  (1110) edge[] (1111)
  
  ;
  \end{tikzpicture}
  \caption{An example pair $f, g: \B^4 \to \B$ with $f \prec g$. Dark grey is $\False(f) \cap \False(g)$, light grey is $\False(f) \cap \Truth(g)$, and white is $\Truth(f) \cap \Truth(g)$. Nodes are labels with $y_1y_2y_3y_4$. The pair $(f,g)$ is $K$-jointly realizable, but not $\Sigma\Pi\Sigma$-jointly realizable. Visually, the inner cube is the floor in the fourth direction, and the outer cube is the ceiling in the fourth direction.}
  \label{fig:eDSGRNsubsetK}
  \end{figure}
  
\begin{lem}
Let $n\geq 4$. There exists a pair $f\prec g \in \MBF^+(n)$ such that $(f,g)$ is not $\Sigma\Pi\Sigma$-jointly realizable, but $(f,g)$ is $K$-jointly realizable. 
\end{lem}

\begin{proof}
Recall that any pair $f\prec g \in \MBF^+(n)$ is $K$-jointly realizable by Theorem~\ref{thm:KCompatible}. It remains to construct an example that is not $\Sigma\Pi\Sigma$-jointly realizable and apply induction. 
We construct an explicit pair for $n=4$. Consider the pair $(f,g)$ in Figure~\ref{fig:eDSGRNsubsetK}. Observe that the floor (inner cube) has the same truth set as the cube in Figure~\ref{fig:DSGRNSubseteDSGRN}. In Lemma~\ref{lem:DSGRNSubseteDSGRN} we showed that this pair can be realized by the $\Sigma\Pi\Sigma$-interaction function $\Lambda= z_1z_2 + z_3$. It turns out this is the only valid $\Sigma\Pi\Sigma$-interaction function that can realize the pair. To see this, we first list all possible $\Sigma\Pi\Sigma$-interaction functions for $n=3$, which are
\[
\begin{matrix}
z_1 + z_2 + z_3 &
(z_1 + z_2) z_3 &
(z_1 + z_3) z_2 &
(z_2 + z_3) z_1 \\ 
z_1 z_2 z_3 &
z_1z_2 + z_3 &
z_1z_3 + z_2 &
z_2z_3 + z_1
\end{matrix}
\]
Via Equations~\eqref{eq:problemDirectCube7_1} and \eqref{eq:problemDirectCube7_2} we can rule out all but $z_1z_2 + z_3$ and $(z_1+z_2)z_3$. In addition, Lemma~\ref{lem:DSGRNSubseteDSGRN} showed that $\Lambda = (z_1+z_2)z_3$ also does not work. Therefore, we can make the following claim: if $(\Lambda\circ \phi,$ \textbf{RN}) $\Sigma\Pi\Sigma$-jointly realizes the pair of MBFs from Figure~\ref{fig:eDSGRNsubsetK}, then $\Lambda = z_1z_2 + z_3$. 

Suppose that $f\prec g \in \MBF^+(4)$ in Figure~\ref{fig:eDSGRNsubsetK} are $\Sigma\Pi\Sigma$-jointly realizable. Then by Theorem~\ref{thm:extendedEmbedding} the floor (inner cube) and ceiling (outer cube) pairs $(f'_{\bF_4}, g'_{\bF_4})$ and $(f'_{\bC_4}, g'_{\bC_4})$ defined rigorously in Theorem~\ref{thm:extendedEmbedding} are $\Sigma\Pi\Sigma$-jointly realizable, and there is a single $\Sigma\Pi\Sigma$-interaction function $\Lambda$ along with maps $\phi_{\bC_4},\phi_{\bF_4}$, realizing networks \textbf{RN}$'_{\bC_4}$, \textbf{RN}$'_{\bF_4}$, and thresholds $\theta_{\bC_4,1}, \theta_{\bC_4,2}, \theta_{\bF_4,1}, \theta_{\bF_4,2}$ such that $(\Lambda\circ \phi_{\bF_4},$ \textbf{RN}$'_{\bF_4}$) $\Sigma\Pi\Sigma$-jointly realizes $(f'_{\bF_4}, g'_{\bF_4})$ and $(\Lambda\circ \phi_{\bC_4},$ \textbf{RN}$'_{\bC_4}$) $\Sigma\Pi\Sigma$-jointly realizes $(f'_{\bC_4}, g'_{\bC_4})$. By our above claim, we know $\Lambda = z_1z_2+z_3$. However, consider $(f'_{\bC_4}, g'_{\bC_4})$. By inspection, we see that
\[ \ttCol_3(\Truth(f'_{\bC_4})|_{\bC_3}) \nsubseteq \ttCol_3(\Truth(g'_{\bC_4})|_{\bF_3})
\] \[
\ttCol_3(\Truth(f'_{\bC_4})|_{\bC_3}) \nsupseteq \ttCol_3\Truth(g'_{\bC_4})|_{\bF_3}) \]
By Theorem~\ref{thm:factorTermComparable} we know that if $ (\Lambda \circ \phi_{\bC_4},$ \textbf{RN}$'_{\bC_4}$) $\Sigma\Pi\Sigma$-jointly  realizes $(f'_{\bC_4}, g'_{\bC_4})$, then $z_3$ cannot be a factor or simple term of $\Lambda$. This contradicts our claim that $\Lambda = z_1z_2+z_3$, and so $(f,g)$ are not $\Sigma\Pi\Sigma$-jointly realizable. 

We now prove the inductive step. Let $n\geq 4$. Assume there exists $f,g : \B^n \to \B$ such that $(f,g)$ is not $\Sigma\Pi\Sigma$-jointly realizable, but is $K$-jointly realizable.
Define 
\[
\tilde f(y_1\dots y_{n+1}) := \begin{cases}
f(y_1\dots x_n) & y_{n+1}=0 \\
1 & y_{n+1}=1
\end{cases}
\]
\[
\tilde g(y_1\dots y_{n+1}) := \begin{cases}
g(y_1\dots y_n) & y_{n+1}=0 \\
1 & y_{n+1}=1
\end{cases}
\]
Observe that $\tilde f', \tilde g' \in \MBF^+(n+1)$ such that $\tilde f' \prec \tilde g'$.
Since $\tilde f'_{\bF_{n+1}}=f$ and $\tilde g'_{\bF_{n+1}}=g$, by Theorem~\ref{thm:extendedEmbedding} we know $(\tilde f, \tilde g)$ is not $\Sigma\Pi\Sigma$-jointly realizable. By Theorem~\ref{thm:KCompatible}, we know all pairs $f \prec g \in \MBF^+(n+1)$ are always $K$-jointly realizable.
\end{proof}

\subsection{$\Sigma\Pi\Sigma$-joint realizability = $K$-joint realizability for $n=3$}\label{SigmaPiSigmaEqualsK}

This is the final result remaining to be proven in Table~\ref{tableOfResults}. 
To find the total number of pairs $(f,g)$ such that $f,g \in \MBF^+(3)$ and $f \prec g$, we use the bijection given in Definition~\ref{defn:bijection}. 
The number of $K$-realizable pairs $(f,g)$ where $f\prec g \in \MBF^+(3)$ is the same as $\abs{\MBF^+(4)}$. In \cite{church1940nunmerical} it was found that $\abs{\MBF^+(4)} = 168$. We used the software DSGRN to find that $150$ MBFs in $\MBF^+(4)$ are $\Sigma$-realizable. The software was also used to computationally check that there are exactly 150 pairs $f\prec g \in \MBF^+(3)$ that are $\Sigma$-jointly realizable. We explicitly constructed the remaining $18$ pairs $(f_i\prec g_i), i = 1, \ldots 18$ that are provably not  $\Sigma$-joint realizable by applying Theorem~\ref{thm:factorTermComparable} to at least one direction in each case. These pairs are all presented in Appendix Table~\ref{table:problemCube} with the direction that allows application of Theorem~\ref{thm:factorTermComparable} indicated. Finally, in Appendix Table~\ref{table:realizeProblemCubes}, we provide specific realizing functions $\Lambda \circ \phi$ and realizing network thresholds $\theta_1,\theta_2$ that $\Sigma\Pi\Sigma$-jointly realize all $18$ pairs.

\section{Discussion}

In this work we linked two classes of  dynamical systems, one a continuous time ordinary differential equation (ODE) model and the other a discrete time 
monotone Boolean function (MBF) model. Both of these classes have been used to model dynamics of gene regulatory networks. 
We show that a very general  class of ODE models with a discontinuous right hand side admits an equivalence relation, such that all ODEs in an equivalence class share the  approximate description of dynamics in terms of the identical state transition graph STG. We then showed that each equivalence class corresponds to a collection of MBFs. The collections of MBFs can be arranged into a parameter graph where edges between the collections indicate a one-step change in one of the MBFs.

After establishing the equivalence between collections of MBFs and equivalence classes of $K$ systems of ODEs,  we pose the question about what restrictions, if any, the  algebraic form of  the right-hand side of the ODE imposes on $k$-tuples of MBFs that correspond to realizable equivalence classes of ODEs. 

 We show that the classes of pairs of MBFs with three inputs that are realizable as linear functions are a strict subset of $\Pi \Sigma$-jointly realizable pairs, which is in turn a strict subset of  $\Sigma \Pi \Sigma$-jointly realizable pairs of MBFs. 
We also show that there are pairs of MBFs with any $n\geq 4$ inputs that are $K$-jointly realizable, but are not $\Sigma \Pi \Sigma$ realizable. 
To summarize,the increased complexity of the algebraic expression provides a richer class of models as measured by the set of MBFs that can be realized in a parameter graph. 

Our work opens up many interesting questions about the joint realizability of collections of MBFs.
We will briefly discuss two potential sets of questions.
First, we defined an infinite nested set of classes of nonlinearities. While we only discussed the first three $ \boldsymbol \Sigma \subsetneq \boldsymbol \Pi \boldsymbol \Sigma \subsetneq \boldsymbol \Sigma \boldsymbol \Pi \boldsymbol \Sigma$,  adding alternating products and sums creates larger and larger classes of functions. Do our results extend in this direction?  In other words, are there pairs of monotone Boolean functions that are realizable in a parameter graph via class  $s+1$, but are not realizable in class $s$? Furthermore, is it possible that there are pairs of MBFs that are not realizable in any of the infinite progression of algebraic classes with alternating products and sums, but are $K$-jointly realizable? 

The second class of questions generalizes pairs of monotone Boolean functions to $k$-tuples,   
Our results derive some constraints for pairs of MBFs, which are then used to prove the main results about differences in 
$\ast$-joint realizability. While these results apply pairwise to any $k$-tuple of
MBFs with $f_1 \prec f_2 \prec \ldots \prec f_k$, to rule out realizability of some tuples, we do not know if  there any additional constraints that arise from 
considering, say, triples of functions $f \prec g \prec h$, or $k$-tuples of MBFs. 

By providing a link between a class of discontinuous differential equations and the collection of $k$-tuples of MBFs, this paper provides an opening to a class of new problems in the field of monotone Boolean functions. 

\appendix
\section{Proofs of Theorem~\ref{thm:factorTermComparable} and ~\ref{thm:extendedEmbedding}.}\label{appendixProofs}

\begin{proof}[Proof of Theorem~\ref{thm:factorTermComparable}]
We will proceed by contradiction. Suppose $z_\ell$ is a factor or simple term and suppose the negation of Equations~\eqref{test} and \eqref{test2}. The negation of \eqref{test} is there exists some point $\vec y \in \bF_\ell$ such that $g(\vec y) = 0$ and $f(\vec y + \hat e_\ell) = 1$. The negation of \eqref{test2} implies there exists some point $\vec w \in \bF_\ell$ such that $g(\vec w) = 1$ and $f(\vec w + \hat e_\ell) = 0$. Let $\theta_f$ and $\theta_g$ be the thresholds from \textbf{RN} associated to the two functions respectively. From the definition of $\ast$-jointly realizable, we know
\begin{align}
g(\vec y) = 0 &\iff \Lambda(\phi(\vec y)) < \theta_g \label{14} \\ 
f(\vec y + \hat e_\ell) = 1 &\iff \Lambda(\phi(\vec y+ \hat e_\ell)) > \theta_f  \label{15}\\
g(\vec w) = 1 &\iff \Lambda(\phi(\vec w)) > \theta_g \label{16}\\ 
f(\vec w + \hat e_\ell) = 0 &\iff \Lambda(\phi(\vec w+ \hat e_\ell)) < \theta_f \label{17}
\end{align}
We now proceed by cases.

\textbf{Case 1: ($z_\ell$ is a factor)}
Recall from the definition of factor that there exists a function $\Lambda'$ that does not depend on $z_\ell$ such that $\Lambda = z_\ell\Lambda'$. We have $z_\ell = \phi_\ell(v_\ell)$ for any $\vec v \in \B^n$. 
Taking $\vec v \in \bF_\ell$ we have
\begin{align*}
  \Lambda(\phi(\vec v)) &= \phi_\ell(0) \Lambda'(\phi'(\ttCol_\ell(\vec v))) \\
  \Lambda(\phi(\vec v + \hat e_\ell)) &= \phi_\ell(1) \Lambda'(\phi'(\ttCol_\ell(\vec v))).
\end{align*}
We used the collapse operation because $\Lambda'$ is independent of $z_\ell$ and we took $\phi'(\ttCol_\ell(\vec v)) = (\phi_1(v_1),\dots,\phi_{\ell-1}(v_{\ell-1}),\phi_{\ell+1}(v_{\ell+1}),\dots,\phi_n(v_n))$. Lastly, we used $\ttCol_\ell(\vec v) = \ttCol_\ell(\vec v+\hat e_\ell)$.
We conclude that for any $\vec v \in \bF_\ell$
\[
\frac{\phi_\ell(1)}{\phi_\ell(0)} \Lambda(\phi(\vec v)) = \Lambda(\phi(\vec v + \hat e_\ell)).
\]
From~\eqref{15} and~\eqref{17}, we may write
\begin{align*}
\frac{\phi_\ell(1)}{\phi_\ell(0)} \Lambda(\phi(\vec y)) > \theta_f  \iff \Lambda(\phi(\vec y)) > \frac{\phi_\ell(0)}{\phi_\ell(1)} \theta_f\\
\frac{\phi_\ell(1)}{\phi_\ell(0)} \Lambda(\phi(\vec w)) < \theta_f \iff \Lambda(\phi(\vec w)) < \frac{\phi_\ell(0)}{\phi_\ell(1)} \theta_f
\end{align*}
and combining with \eqref{14} and \eqref{16} we obtain
\[
\theta_g > \frac{\phi_\ell(0)}{\phi_\ell(1)} \theta_f \text{ and }
\theta_g < \frac{\phi_\ell(0)}{\phi_\ell(1)} \theta_f
\]
a clear contradiction.

\textbf{Case 2: ($z_\ell$ is a simple term)}
Similar to Case 1, the key fact is, if $z_\ell$ is a simple term, we know for any $\vec v \in \bF_\ell$
\[
\Lambda(\phi(\vec v)) +\left( \phi_\ell(1) - \phi_\ell(0) \right) = \Lambda(\phi(\vec v + \hat e_\ell)),
\]
We then make a similar argument as before. Equations~\eqref{15} and \eqref{17} give
\begin{align*}
\Lambda(\phi(\vec y)) +\left( \phi_\ell(1) - \phi_\ell(0) \right)> \theta_1  \iff \Lambda(\phi(\vec y)) >  \theta_1 -\left( \phi_\ell(1) - \phi_\ell(0) \right)\\
\Lambda(\phi(\vec w)) +\left( \phi_\ell(1) - \phi_\ell(0) \right)< \theta_1 \iff \Lambda(\phi(\vec w)) < \theta_1 -\left( \phi_\ell(1) - \phi_\ell(0) \right)
\end{align*}
and combining with \eqref{14} and \eqref{16} we obtain
\[
\theta_g > \theta_f -\left( \phi_\ell(1) - \phi_\ell(0) \right)\text{ and }
\theta_g < \theta_f -\left( \phi_\ell(1) - \phi_\ell(0) \right)
\]
which is our desired contradiction.
\end{proof}

\begin{proof}[Proof of Theorem~\ref{thm:extendedEmbedding}]

Let $f, g: \B^n \to \B$, with $f \prec g$, be $\ast$-jointly realizable  MBFs, and let $(\Lambda\circ\phi,$ \textbf{RN}) jointly $\ast$-realize $(f,g)$. Let $\theta_f$ and $\theta_g$ be the thresholds associated to the two realizations of these functions.

We seek to construct a $\ast$-interaction function $\Lambda' : \bbR^{n-1}_+ \to \R_+$, a map $\phi': \B^{n-1} \to \R^{n-1}_+$, and a weighted regulatory network \textbf{RN}$'_U$ with thresholds $\theta_f'$ and $\theta_g'$ so that $(\Lambda'\circ\phi',$ \textbf{RN}$'_U$) $\ast$-jointly realizes  $f'_U, g'_U : \B^{n-1} \to \B$. In doing so we will prove a). In each of the following cases, the construction of $\Lambda'$ does not depend on whether $U = \bF_\ell$ or $U = \bC_\ell$, and so b) will follow immediately. 

Recall that we are going to collapse over the $\ell^{\mbox{th}}$ dimension. In the following proof, we will be considering a node in \textbf{RN} that has an incoming edge from node $\ell$ and two targets, one of which is associated to the MBF $f$ and the other to $g$. It will be useful to define the graph \textbf{RN}$'$ to be the network \textbf{RN} without the edge from $\ell$ to the node under consideration. \textbf{RN}$'$ is an intermediate step to the construction of \textbf{RN}$'_U$.

Using the observations in  Remark~\ref{rem:realizabilitySpecialCases}, we note that regardless of the specific value of $\ast$, we can always assume that $\Lambda \in\boldsymbol \Sigma \boldsymbol\Pi\boldsymbol\Sigma $  defined in (\ref{big:class}). 
Therefore there exist sets $W_1,\dots, W_L$, where $W_1,\dots,W_L$ partitions the set $\{1,\dots,n\}$,
and for each $k \in \{1,\dots,L\}$, there exists $M_k \in \N$ so that the sets $V_{k,1}, \dots, V_{k,M_k}$ partition the set $W_k$, such that 
\[
\Lambda = \sum_{W_k} \left( \prod_{V_{k,j}} \left (\sum_{i \in V_{k,j}} z_i \right) \right)
\]
There is exactly one set in the partition, call it $W_p$, such that $\ell \in W_p$. Furthermore, there is exactly one $V_{p,*}$ that contains $\ell$, call it $V_{p,q}$. We now proceed by cases. Define the map $\delta:\{1,\dots,n\}\setminus\{\ell\} \to \{1,\dots,n-1\}$ as
\[
\delta(j) = \begin{cases}
j &\text{if } j < \ell \\
j-1 &\text{if } j > \ell 
\end{cases}
\]
We will use the map $\delta$ to construct the interaction function $\Lambda'$. If $\ast$ is $\Sigma\Pi\Sigma$, we need to consider Cases 1, 2, and 3.
However, if $\ast$ is $\Pi\Sigma$, then $W_1=\{1,\dots,n\}$, and we only need to consider Cases 2 and 3, and if $\ast$ is $\Sigma$, then $W_1=V_{1,1}=\{1,\dots,n\}$, and we only need to consider Case 3.

\textbf{Case 1: ($W_p = \{\ell\}$)} In this case the map $\Lambda$ has the form
\begin{equation*}
\Lambda = z_\ell + \sum_{\substack{W_k \\ k \neq p}} \left( \prod_{V_{k,j}} \left (\sum_{i \in V_{k,j}} z_i \right) \right)
\end{equation*}
where $z_i = \phi_i(y_i)$ for $\vec y \in \B^n$. We construct the interaction function $\Lambda'$ as 
\begin{equation*}
\Lambda' = \sum_{\substack{W_k \\ k \neq p}} \left( \prod_{V_{k,j}} \left (\sum_{i \in V_{k,j}} z_{\delta(i)} \right) \right)
\end{equation*}
and define $z_{\delta(i)} = \phi'_i(y_i) = \phi_{\delta^{-1}(i)}(y_i)$. We then set $\phi' = (\phi_1',\dots,\phi_{n-1}')$. This construction ensures, for any $\vec y = (y_1,\dots,y_n) \in \B^n$, 
\[
\Lambda(\phi(\vec y) )= \Lambda'(\phi'(y_{\delta^{-1}(1)}, \dots, y_{\delta^{-1}(n-1)})) + \phi_\ell(y_\ell)
\]
If $U = \bF_\ell$, then for all $\vec y = (y_1,\dots,y_n) \in U$, we know $y_\ell = 0$. Therefore, 
\[ 
\Lambda(\phi(\vec y) )= \Lambda'(\phi'(y_{\delta^{-1}(1)}, \dots, y_{\delta^{-1}(n-1)})) + \phi_\ell(0).
\]

Set $\theta_f' = \max\{0,\theta_f-\phi_\ell(0)\}$ and $\theta_g' = \max\{0,\theta_g-\phi_\ell(0)\}$. It is possible that $\theta_f' = \theta_g'$ at this point. However, since $\Lambda \circ \phi$ takes on finitely many values, we can always perturb one threshold by a small enough $\ep$ to guarantee our inequalities still hold. After this potential perturbation, replace $\theta_f$ and $\theta_g$ in \textbf{RN}$'$ with $\theta_f'$ and $\theta_g'$ to complete the construction of \textbf{RN}$'_{\bF_\ell}$. This construction means that 
\[  \Lambda(\phi(y_1,\dots,y_n) ) \lessgtr \theta_f \iff  \Lambda'(\phi'(y_{\delta^{-1}(1)}, \dots, y_{\delta^{-1}(n-1)}) ) \lessgtr \theta'_f \]
and likewise for $\theta_g$ and $\theta_g'$. Therefore $(\Lambda' \circ \phi',$ \textbf{RN}$'_{\bF_\ell}$) $*$-jointly realizes $f'_{\bF_\ell}$ and $g'_{\bF_\ell}$.

Similarly, if $U = \bC_\ell$, then for all $\vec y = (y_1,\dots,y_n) \in U$, we have
\[ 
\Lambda(\phi(\vec y) )= \Lambda'(\phi'(y_{\delta^{-1}(1)}, \dots, y_{\delta^{-1}(n-1)})) + \phi_\ell(1).
\]
We set $\theta_f' = \max\{0,\theta_f-\phi_\ell(1)\}$ and $\theta_g' = \max\{0,\theta_g-\phi_\ell(1)\}$, perturbed by small enough $\ep>0$, if necessary,  to replace $\theta_f$ and $\theta_g$ in \textbf{RN}$'$ and complete the construction of \textbf{RN}$'_{\bC_\ell}$. Then $(\Lambda' \circ \phi',$ \textbf{RN}$'_{\bC_\ell}$) $*$-jointly realizes $f'_{\bC_\ell}$ and $g'_{\bC_\ell}$.

\textbf{Case 2: ($W_p \setminus \{\ell\} \neq \emptyset$ and $V_{p,q} = \{\ell\}$) } The interaction function $\Lambda'$ without the $\ell^{\mbox{th}}$ element is 
\[
  \Lambda' = \sum_{W_k} \left( \prod_{\substack{V_{k,j} \\ j \neq q}} \left (\sum_{i \in V_{k,j}} z_{\delta(i)} \right) \right)
\]
In this case we know that $W_p$ is a partition of at minimum size two. Pick exactly one $t \in W_p \setminus \{q\}$. 
Construct $z_{\delta(j)} = \phi'_j(y_j)$ as follows: if $U = \bF_i$, then for $\vec y \in \B^n$
\[
\phi_j' (y_j) = \begin{cases}
\phi_{\delta^{-1}(j)} (y_j)  \phi_\ell(0) & \text{if } j \in V_{p,t} \\
\phi_{\delta^{-1}(j)} (y_j) & \text{otherwise}
\end{cases} \quad .
\]
However, if $U = \bC_i$, then
\[
\phi_j' (y_j) = \begin{cases}
\phi_{\delta^{-1}(j)} (y_j)  \phi_\ell(1) & \text{if } j \in V_{p,t} \\
\phi_{\delta^{-1}(j)} (y_j) & \text{otherwise}
\end{cases} \quad .
\]
We have ensured for any $\vec y = (y_1,\dots,y_n) \in U$, 
\[ 
\Lambda(\phi(\vec y) )= \Lambda'(\phi'(y_{\delta^{-1}(1)}, \dots, y_{\delta^{-1}(n-1)}))
\]
Setting $\theta_f' = \theta_f$ and $\theta_g' = \theta_g$ obtains the desired result; i.e., \textbf{RN}$'_U$ = \textbf{RN}$'$ and $(\Lambda' \circ \phi',$\textbf{RN}$')$ $*$-jointly realizes $f'_U, g'_U$.

\textbf{Case 3: ($W_p \setminus \{\ell\} \neq \emptyset$ and $V_{p,q} \setminus \{\ell\} \neq \emptyset$) }
In this case the original interaction function $\Lambda$ takes the form
\[
\Lambda = \sum_{\substack{W_k}} \left( \prod_{\substack{V_{k,j}  \\ (k,j) \neq (p,q)}} \left (\sum_{\substack{i \in V_{k,j}}} z_i \right) \right) + \left( \prod_{V_{p,q}} \left (\sum_{i \in V_{p,q}} z_i \right) \right) \ .
\]
The interaction function $\Lambda'$ is constructed as 
\[
\Lambda' = \sum_{\substack{W_k}} \left( \prod_{\substack{V_{k,j}  \\ (k,j) \neq (p,q)}} \left (\sum_{i \in V_{k,j}} z_{\delta(i)} \right) \right) + \left( \prod_{V_{p,q}} \left (\sum_{\substack{i \in V_{p,q} \\ i \neq \ell}} z_{\delta(i)}  \right) \right) \ .
\]
Pick exactly one element $t \in V_{p,q} \setminus \{\ell\}$, and construct $\phi'_j$ as follows: if $U = \bF_\ell$, then
\[
\phi_j' (y_j) := \begin{cases}
\phi_{\delta^{-1}(j)} (y_j) +  \phi_\ell(0) & \text{if } j = \delta(t) \\
\phi_{\delta^{-1}(j)} (y_j) & \text{otherwise}
\end{cases} \quad .
\]
However, if $U = \bC_\ell$, then
\[
\phi_j' (y_j) := \begin{cases}
\phi_{\delta^{-1}(j)} (y_j) +  \phi_\ell(1) & \text{if } j = \delta(t) \\
\phi_{\delta^{-1}(j)} (y_j) & \text{otherwise}
\end{cases} \quad .
\]
We then set $\phi' = (\phi_1',\dots,\phi_{n-1}')$. This construction ensures, for any $\vec y = (y_1,\dots,y_n) \in U$, 
\[ 
\Lambda(\phi(\vec y) )= \Lambda'(\phi'(y_{\delta^{-1}(1)}, \dots, y_{\delta^{-1}(n-1)}))
\]
and so by setting \textbf{RN}$'_U=$ \textbf{RN}$'$ as in Case 2,  we obtain the desired result.
\end{proof}

\section{Supporting tables and figures}\label{problemCubes}

Table~\ref{table:problemCube} lists explicitly all pairs $f\prec g \in {\bf MBF^+(3)}$ of Boolean functions with three inputs that are not $\Sigma$-jointly realizable. These correspond to non-threshold monotone Boolean functions in $\MBF^+(4)$ with $4$ inputs. In each case the direction that allows us to use Theorem~\ref{thm:factorTermComparable} to rule out $\Sigma$-joint realizability is indicated in the last column. 

\begin{table}[htp]
\begin{center}
\begin{tabular}{|c|c|c|c|c|c|c|c|c|c|}  \cline{2-9}
\multicolumn{1}{l|}{} & \multicolumn{8}{c|}{Input} & \multicolumn{1}{c}{} \\ \cline{2-10} 
\multicolumn{1}{l|}{} & 000 & 001 & 010 & 100 & 110 & 101 & 011 & 111 & \multicolumn{1}{c|}{Direction(s)} \\ \hline
$f_1 \prec g_1$ & 00  & 01  & 00  & 01  & 11  & 01  & 01  & 11 & $y_1$ \\ 
$f_2 \prec g_2$ & 00  & 00  & 01  & 01  & 01  & 01  & 11  & 11 & $y_2$ \\    
$f_3 \prec g_3$ & 00  & 01  & 01  & 00  & 01  & 11  & 01  & 11 & $y_3$ \\ 
$f_4 \prec g_4$ & 00  & 01  & 01  & 00  & 11  & 01  & 01  & 11 & $y_2$ \\ 
$f_5 \prec g_5$ & 00  & 01  & 00  & 01  & 01  & 01  & 11  & 11 & $y_3$ \\ 
$f_6 \prec g_6$ & 00  & 00  & 01  & 01  & 01  & 11  & 01  & 11 & $y_1$ \\ 
$f_7 \prec g_7$ & 00  & 01  & 00  & 00  & 11  & 01  & 01  & 11 & $y_1,y_2$ \\ 
$f_8 \prec g_8$ & 00  & 00  & 00  & 01  & 01  & 01  & 11  & 11 & $y_2,y_3$ \\ 
$f_9 \prec g_9$ & 00  & 00  & 01  & 00  & 01  & 11  & 01  & 11 & $y_1,y_3$ \\ 
$f_{10} \prec g_{10}$ & 00  & 01  & 00  & 01  & 11  & 01  & 11  & 11 & $y_1,y_3$ \\ 
$f_{11} \prec g_{11}$ & 00  & 00  & 01  & 01  & 01  & 11  & 11  & 11 & $y_1,y_2$ \\ 
$f_{12} \prec g_{12}$ & 00  & 01  & 01  & 00  & 11  & 11  & 01  & 11 & $y_2,y_3$ \\ 
$f_{13} \prec g_{13}$ & 00  & 01  & 00  & 00  & 11  & 01  & 11  & 11 & $y_1$ \\ 
$f_{14} \prec g_{14}$ & 00  & 00  & 00  & 01  & 01  & 11  & 11  & 11 & $y_2$ \\ 
$f_{15} \prec g_{15}$ & 00  & 00  & 01  & 00  & 11  & 11  & 01  & 11 & $y_3$ \\ 
$f_{16} \prec g_{16}$ & 00  & 00  & 00  & 01  & 11  & 01  & 11  & 11 & $y_1$ \\ 
$f_{17} \prec g_{17}$ & 00  & 00  & 01  & 00  & 01  & 11  & 11  & 11 & $y_2$ \\ 
$f_{18} \prec g_{18}$ & 00  & 01  & 00  & 00  & 11  & 11  & 01  & 11 & $y_3$ \\ \hline
\end{tabular}
\end{center}
\caption{All $18$ non $\Sigma$-jointly realizable pairs $(f,g)$ for $n=3$ such that $f \prec g$. The column for input $\vec y$ shows the pair of values of $f_i(y)g_i(y)$. For example, column 001 has 01 in the first row. This means that $f_1(001) = 0$ and $g_1(001) = 1$. The direction(s) that allow the use of Theorem~\ref{thm:factorTermComparable} to rule out $\Sigma$-joint realizability are indicated in the last column.}\label{table:problemCube}
\end{table}

In Table~\ref{table:realizeProblemCubes} we show explicitly the form of the interaction function $\Lambda \in \boldsymbol \Sigma  \boldsymbol\Pi  \boldsymbol\Sigma$,
the values $\phi(1) = (\phi_1(1),\phi_2(1),\phi_3(1))$, and the value of the thresholds $\theta_g, \theta_f$ that $\Sigma\Pi\Sigma$-jointly realize all pairs $f_i\prec g_i$ given in Table~\ref{table:problemCube}. In all cases, we set $\phi_1(0)=\phi_2(0)=\phi_3(0)=1$. This list, together with $150$ pairs that are $\Sigma$-jointly realizable, exhausts all pairs $f\prec g$ of functions in ${\MBF^+(3)}$ and proves that for $n=3$ every such pair is $\Sigma\Pi\Sigma$-joint realizable. 

Figure~\ref{fig:megaPG} shows all pairs $f \prec g \in \MBF^+(2)$. It is also the factor of the parameter graph associated to node 1 in the network in Figure~\ref{fig:exampleRN}, after transforming the Boolean functions to be positive under the map $\beta$ given in \eqref{eq:beta}.

\begin{table}[htp!]
\begin{center}
\begin{tabular}{|c|c|c|c|c|} \hline
$i$ & $\Lambda$ & $\phi(111)$ & $\theta_g$ & $\theta_f$\\ \hline
1 & $z_1z_2+z_3$ & $(3,2,3)$ & 3.5 & 6.5\\ \hline
2 & $z_1+z_2z_3$ & $(3,3,2)$ & 3.5 & 6.5\\ \hline
3 & $z_2+z_1z_3$ & $(2,3,3)$ & 3.5 & 6.5\\ \hline
4 & $z_1z_2+z_3$ & $(2,3,3)$ & 3.5 & 6.5\\ \hline
5 & $z_1+z_2z_3$ & $(3,2,3)$ & 3.5 & 6.5\\ \hline
6 & $z_2+z_1z_3$ & $(3,3,2)$ & 3.5 & 6.5\\ \hline
7 & $z_1z_2+z_3$ & $(3,3,4)$ & 4.5 & 8\\ \hline
8 & $z_1+z_2z_3$ & $(4,3,3)$ & 4.5 & 8\\ \hline
9 & $z_2+z_1z_3$ & $(3,4,3)$ & 4.5 & 8\\ \hline
10 & $z_2(z_1+z_3)$ & $(4,2,4)$ & 4.5 & 9\\ \hline
11 & $z_3(z_1+z_2)$ & $(4,4,2)$ & 4.5 & 9\\ \hline
12 & $z_1(z_2+z_3)$ & $(2,4,4)$ & 4.5 & 9\\ \hline
13 & $z_1z_2+z_3$ & $(2,3,4)$ & 4.5 & 6.5\\ \hline
14 & $z_1+z_2z_3$ & $(4,2,3)$ & 4.5 & 6.5\\ \hline
15 & $z_2+z_1z_3$ & $(3,4,2)$ & 4.5 & 6.5\\ \hline
16 & $z_2(z_1+z_3)$ & $(4,2,3)$ & 4.5 & 7.5\\ \hline
17 & $z_3(z_1+z_2)$ & $(3,4,2)$ & 4.5 & 7.5\\ \hline
18 & $z_1(z_2+z_3)$ & $(2,3,4)$ & 4.5 & 7.5\\ \hline
\end{tabular}
\end{center}
\caption{Example $(\Lambda\circ \phi, \theta_1,\theta_2)$ that $\Sigma\Pi\Sigma$-jointly realize all pairs $f_i\prec g_i$ given in Table~\ref{table:problemCube}. For all rows, $\phi_1(0)=\phi_2(0)=\phi_3(0)=1$.}
\label{table:realizeProblemCubes}
\end{table}%

\begin{figure}[htp!]
    \centering
   \begin{tikzpicture}[node distance = .3cm]
\tikzset{my node/.style = {circle, draw, very thick, inner sep=0pt,minimum size=0pt}}
	\node[my node] (1) {
	\begin{tikzpicture}[node distance = .2cm]
\tikzset{my node/.style = {circle, draw, inner sep=1pt,minimum size=0pt}}
	\node[my node] (A) {\small 00};
	\node[my node] (B) [right=of A] {\small 01};
	\node[my node] (C) [above=of A] {\small 10};
	\node[my node] (D) [right=of C] {\small 11};

 	\draw[-, thick] [] (A) to [] (B);
    \draw[-, thick] [] (A) to [] (C);
    \draw[-, thick] [] (C) to [] (D);
    \draw[-, thick] [] (B) to [] (D);
\end{tikzpicture}
};
	\node[my node] (2) [below=of 1] {
	\begin{tikzpicture}[ node distance = .2cm]
\tikzset{my node/.style = {circle, draw, inner sep=1pt,minimum size=0pt}}
	\node[my node, fill=lightgray] (A) {\small 00};
	\node[my node] (B) [right=of A] {\small 01};
	\node[my node] (C) [above=of A] {\small 10};
	\node[my node] (D) [right=of C] {\small 11};

 	\draw[-, thick] [] (A) to [] (B);
    \draw[-, thick] [] (A) to [] (C);
    \draw[-, thick] [] (C) to [] (D);
    \draw[-, thick] [] (B) to [] (D);
\end{tikzpicture}
};
	\node[my node] (3) [below=of 2] {
	\begin{tikzpicture}[ node distance = .2cm]
\tikzset{my node/.style = {circle, draw, inner sep=1pt,minimum size=0pt}}
	\node[my node, fill=darkgray] (A) {\small 00};
	\node[my node] (B) [right=of A] {\small 01};
	\node[my node] (C) [above=of A] {\small 10};
	\node[my node] (D) [right=of C] {\small 11};

 	\draw[-, thick] [] (A) to [] (B);
    \draw[-, thick] [] (A) to [] (C);
    \draw[-, thick] [] (C) to [] (D);
    \draw[-, thick] [] (B) to [] (D);
\end{tikzpicture}
};
	\node[my node] (4) [left=of 3] {
	\begin{tikzpicture}[ node distance = .2cm]
\tikzset{my node/.style = {circle, draw, inner sep=1pt,minimum size=0pt}}
	\node[my node, fill=lightgray] (A) {\small 00};
	\node[my node, fill=lightgray] (B) [right=of A] {\small 01};
	\node[my node] (C) [above=of A] {\small 10};
	\node[my node] (D) [right=of C] {\small 11};

 	\draw[-, thick] [] (A) to [] (B);
    \draw[-, thick] [] (A) to [] (C);
    \draw[-, thick] [] (C) to [] (D);
    \draw[-, thick] [] (B) to [] (D);
\end{tikzpicture}
};
	\node[my node] (5) [right=of 3] {
	\begin{tikzpicture}[ node distance = .2cm]
\tikzset{my node/.style = {circle, draw, inner sep=1pt,minimum size=0pt}}
	\node[my node, fill=lightgray] (A) {\small 00};
	\node[my node] (B) [right=of A] {\small 01};
	\node[my node, fill=lightgray] (C) [above=of A] {\small 10};
	\node[my node] (D) [right=of C] {\small 11};

 	\draw[-, thick] [] (A) to [] (B);
    \draw[-, thick] [] (A) to [] (C);
    \draw[-, thick] [] (C) to [] (D);
    \draw[-, thick] [] (B) to [] (D);
\end{tikzpicture}
};
	\node[my node] (6) [below=of 3] {
	\begin{tikzpicture}[node distance = .2cm]
\tikzset{my node/.style = {circle, draw, inner sep=1pt,minimum size=0pt}}
	\node[my node, fill=lightgray] (A) {\small 00};
	\node[my node, fill=lightgray] (B) [right=of A] {\small 01};
	\node[my node, fill=lightgray] (C) [above=of A] {\small 10};
	\node[my node] (D) [right=of C] {\small 11};

 	\draw[-, thick] [] (A) to [] (B);
    \draw[-, thick] [] (A) to [] (C);
    \draw[-, thick] [] (C) to [] (D);
    \draw[-, thick] [] (B) to [] (D);
\end{tikzpicture}
	};
	\node[my node] (7) [left=of 6] {
	\begin{tikzpicture}[node distance = .2cm]
\tikzset{my node/.style = {circle, draw, inner sep=1pt,minimum size=0pt}}
	\node[my node, fill=darkgray] (A) {\small 00};
	\node[my node, fill=lightgray] (B) [right=of A] {\small 01};
	\node[my node, fill=white] (C) [above=of A] {\small 10};
	\node[my node, fill=white] (D) [right=of C] {\small 11};

 	\draw[-, thick] [] (A) to [] (B);
    \draw[-, thick] [] (A) to [] (C);
    \draw[-, thick] [] (C) to [] (D);
    \draw[-, thick] [] (B) to [] (D);
\end{tikzpicture}
};
	\node[my node] (8) [right=of 6] {
	\begin{tikzpicture}[node distance = .2cm]
\tikzset{my node/.style = {circle, draw, inner sep=1pt,minimum size=0pt}}
	\node[my node, fill=darkgray] (A) {\small 00};
	\node[my node, fill=white] (B) [right=of A] {\small 01};
	\node[my node, fill=lightgray] (C) [above=of A] {\small 10};
	\node[my node, fill=white] (D) [right=of C] {\small 11};

 	\draw[-, thick] [] (A) to [] (B);
    \draw[-, thick] [] (A) to [] (C);
    \draw[-, thick] [] (C) to [] (D);
    \draw[-, thick] [] (B) to [] (D);
\end{tikzpicture}
};
	\node[my node] (9) [below left=0.9cm and .5cm of 6] {
	\begin{tikzpicture}[node distance = .2cm]
\tikzset{my node/.style = {circle, draw, inner sep=1pt,minimum size=0pt}}
	\node[my node, fill=lightgray] (A) {\small 00};
	\node[my node, fill=lightgray] (B) [right=of A] {\small 01};
	\node[my node, fill=lightgray] (C) [above=of A] {\small 10};
	\node[my node, fill=lightgray] (D) [right=of C] {\small 11};

 	\draw[-, thick] [] (A) to [] (B);
    \draw[-, thick] [] (A) to [] (C);
    \draw[-, thick] [] (C) to [] (D);
    \draw[-, thick] [] (B) to [] (D);
\end{tikzpicture}
};
	\node[my node] (10) [right=of 9] {
	\begin{tikzpicture}[node distance = .2cm]
\tikzset{my node/.style = {circle, draw, inner sep=1pt,minimum size=0pt}}
	\node[my node, fill=darkgray] (A) {\small 00};
	\node[my node, fill=lightgray] (B) [right=of A] {\small 01};
	\node[my node, fill=lightgray] (C) [above=of A] {\small 10};
	\node[my node, fill=white] (D) [right=of C] {\small 11};

 	\draw[-, thick] [] (A) to [] (B);
    \draw[-, thick] [] (A) to [] (C);
    \draw[-, thick] [] (C) to [] (D);
    \draw[-, thick] [] (B) to [] (D);
\end{tikzpicture}
};
	\node[my node] (11) [right=of 10] {\begin{tikzpicture}[node distance = .2cm]
\tikzset{my node/.style = {circle, draw, inner sep=1pt,minimum size=0pt}}
	\node[my node, fill=darkgray] (A) {\small 00};
	\node[my node, fill=white] (B) [right=of A] {\small 01};
	\node[my node, fill=darkgray] (C) [above=of A] {\small 10};
	\node[my node, fill=white] (D) [right=of C] {\small 11};

 	\draw[-, thick] [] (A) to [] (B);
    \draw[-, thick] [] (A) to [] (C);
    \draw[-, thick] [] (C) to [] (D);
    \draw[-, thick] [] (B) to [] (D);
\end{tikzpicture}
};
	\node[my node] (12) [left=of 9] {\begin{tikzpicture}[node distance = .2cm]
\tikzset{my node/.style = {circle, draw, inner sep=1pt,minimum size=0pt}}
	\node[my node, fill=darkgray] (A) {\small 00};
	\node[my node, fill=darkgray] (B) [right=of A] {\small 01};
	\node[my node, fill=white] (C) [above=of A] {\small 10};
	\node[my node, fill=white] (D) [right=of C] {\small 11};

 	\draw[-, thick] [] (A) to [] (B);
    \draw[-, thick] [] (A) to [] (C);
    \draw[-, thick] [] (C) to [] (D);
    \draw[-, thick] [] (B) to [] (D);
\end{tikzpicture}
};
	\node[my node] (13) [below right=0.9cm and .5cm of 9] {
	\begin{tikzpicture}[node distance = .2cm]
\tikzset{my node/.style = {circle, draw, inner sep=1pt,minimum size=0pt}}
	\node[my node, fill=darkgray] (A) {\small 00};
	\node[my node, fill=lightgray] (B) [right=of A] {\small 01};
	\node[my node, fill=lightgray] (C) [above=of A] {\small 10};
	\node[my node, fill=lightgray] (D) [right=of C] {\small 11};

 	\draw[-, thick] [] (A) to [] (B);
    \draw[-, thick] [] (A) to [] (C);
    \draw[-, thick] [] (C) to [] (D);
    \draw[-, thick] [] (B) to [] (D);
\end{tikzpicture}
};
	\node[my node] (14) [right=of 13] {
	\begin{tikzpicture}[node distance = .2cm]
\tikzset{my node/.style = {circle, draw, inner sep=1pt,minimum size=0pt}}
	\node[my node, fill=darkgray] (A) {\small 00};
	\node[my node, fill=lightgray] (B) [right=of A] {\small 01};
	\node[my node, fill=darkgray] (C) [above=of A] {\small 10};
	\node[my node, fill=white] (D) [right=of C] {\small 11};

 	\draw[-, thick] [] (A) to [] (B);
    \draw[-, thick] [] (A) to [] (C);
    \draw[-, thick] [] (C) to [] (D);
    \draw[-, thick] [] (B) to [] (D);
\end{tikzpicture}
};
	\node[my node] (15) [left=of 13] {
	\begin{tikzpicture}[node distance = .2cm]
\tikzset{my node/.style = {circle, draw, inner sep=1pt,minimum size=0pt}}
	\node[my node, fill=darkgray] (A) {\small 00};
	\node[my node, fill=darkgray] (B) [right=of A] {\small 01};
	\node[my node, fill=lightgray] (C) [above=of A] {\small 10};
	\node[my node, fill=white] (D) [right=of C] {\small 11};

 	\draw[-, thick] [] (A) to [] (B);
    \draw[-, thick] [] (A) to [] (C);
    \draw[-, thick] [] (C) to [] (D);
    \draw[-, thick] [] (B) to [] (D);
\end{tikzpicture}
};
	\node[my node] (16) [below=of 13] {
	\begin{tikzpicture}[node distance = .2cm]
\tikzset{my node/.style = {circle, draw, inner sep=1pt,minimum size=0pt}}
	\node[my node, fill=darkgray] (A) {\small 00};
	\node[my node, fill=darkgray] (B) [right=of A] {\small 01};
	\node[my node, fill=darkgray] (C) [above=of A] {\small 10};
	\node[my node, fill=white] (D) [right=of C] {\small 11};

 	\draw[-, thick] [] (A) to [] (B);
    \draw[-, thick] [] (A) to [] (C);
    \draw[-, thick] [] (C) to [] (D);
    \draw[-, thick] [] (B) to [] (D);
\end{tikzpicture}
};
	\node[my node] (17) [right=of 16] {
	\begin{tikzpicture}[node distance = .2cm]
\tikzset{my node/.style = {circle, draw, inner sep=1pt,minimum size=0pt}}
	\node[my node, fill=darkgray] (A) {\small 00};
	\node[my node, fill=lightgray] (B) [right=of A] {\small 01};
	\node[my node, fill=darkgray] (C) [above=of A] {\small 10};
	\node[my node, fill=lightgray] (D) [right=of C] {\small 11};

 	\draw[-, thick] [] (A) to [] (B);
    \draw[-, thick] [] (A) to [] (C);
    \draw[-, thick] [] (C) to [] (D);
    \draw[-, thick] [] (B) to [] (D);
\end{tikzpicture}
};
	\node[my node] (18) [left=of 16] {
	\begin{tikzpicture}[node distance = .2cm]
\tikzset{my node/.style = {circle, draw, inner sep=1pt,minimum size=0pt}}
	\node[my node, fill=darkgray] (A) {\small 00};
	\node[my node, fill=darkgray] (B) [right=of A] {\small 01};
	\node[my node, fill=lightgray] (C) [above=of A] {\small 10};
	\node[my node, fill=lightgray] (D) [right=of C] {\small 11};

 	\draw[-, thick] [] (A) to [] (B);
    \draw[-, thick] [] (A) to [] (C);
    \draw[-, thick] [] (C) to [] (D);
    \draw[-, thick] [] (B) to [] (D);
\end{tikzpicture}
};
	\node[my node] (19) [below=of 16] {
	\begin{tikzpicture}[node distance = .2cm]
\tikzset{my node/.style = {circle, draw, inner sep=1pt,minimum size=0pt}}
	\node[my node, fill=darkgray] (A) {\small 00};
	\node[my node, fill=darkgray] (B) [right=of A] {\small 01};
	\node[my node, fill=darkgray] (C) [above=of A] {\small 10};
	\node[my node, fill=lightgray] (D) [right=of C] {\small 11};

 	\draw[-, thick] [] (A) to [] (B);
    \draw[-, thick] [] (A) to [] (C);
    \draw[-, thick] [] (C) to [] (D);
    \draw[-, thick] [] (B) to [] (D);
\end{tikzpicture}
};
	\node[my node] (20) [below=of 19] {
	\begin{tikzpicture}[node distance = .2cm]
\tikzset{my node/.style = {circle, draw, inner sep=1pt,minimum size=0pt}}
	\node[my node, fill=darkgray] (A) {\small 00};
	\node[my node, fill=darkgray] (B) [right=of A] {\small 01};
	\node[my node, fill=darkgray] (C) [above=of A] {\small 10};
	\node[my node, fill=darkgray] (D) [right=of C] {\small 11};

 	\draw[-, thick] [] (A) to [] (B);
    \draw[-, thick] [] (A) to [] (C);
    \draw[-, thick] [] (C) to [] (D);
    \draw[-, thick] [] (B) to [] (D);
\end{tikzpicture}
};
	
\tikzset{my edge/.style = {-, very thick}}	
  \draw[my edge] [] (1) to [] (2);
  \draw[my edge] [] (2) to [] (4);
  \draw[my edge] [] (2) to [] (3);
  \draw[my edge] [] (2) to [] (5);
  \draw[my edge] [] (4) to [] (7);
  \draw[my edge] [] (4) to [] (6);
  \draw[my edge] [] (3) to [] (7);
  \draw[my edge] [] (3) to [] (8);
  \draw[my edge] [] (5) to [] (6);
  \draw[my edge] [] (5) to [] (8);
  \draw[my edge] [] (7) to [] (12);
  \draw[my edge] [] (7) to [] (10);
  \draw[my edge] [] (6) to [] (9);
  \draw[my edge] [] (6) to [] (10);
  \draw[my edge] [] (8) to [] (10);
  \draw[my edge] [] (8) to [] (11);
  \draw[my edge] [] (12) to [] (15);
  \draw[my edge] [] (9) to [] (13);
  \draw[my edge] [] (11) to [] (14);
  \draw[my edge] [] (10) to [] (15);
  \draw[my edge] [] (10) to [] (13);
  \draw[my edge] [] (10) to [] (14);
  \draw[my edge] [] (15) to [] (18);
  \draw[my edge] [] (15) to [] (16);
  \draw[my edge] [] (13) to [] (18);
  \draw[my edge] [] (13) to [] (17);
  \draw[my edge] [] (14) to [] (16);
  \draw[my edge] [] (14) to [] (17);
  \draw[my edge] [] (18) to [] (19);
  \draw[my edge] [] (16) to [] (19);
  \draw[my edge] [] (17) to [] (19);
  \draw[my edge] [] (19) to [] (20);
\end{tikzpicture}
    \caption{All $20$ pairs of functions $f \prec g \in \MBF^+(2)$. Each node in the above graph represents a pair of functions, where dark grey is $\False(f) \cap \False(g)$, light grey is $\False(f) \cap \Truth(g)$, and white is $\Truth(f) \cap \Truth(g)$.}
    \label{fig:megaPG}
\end{figure}
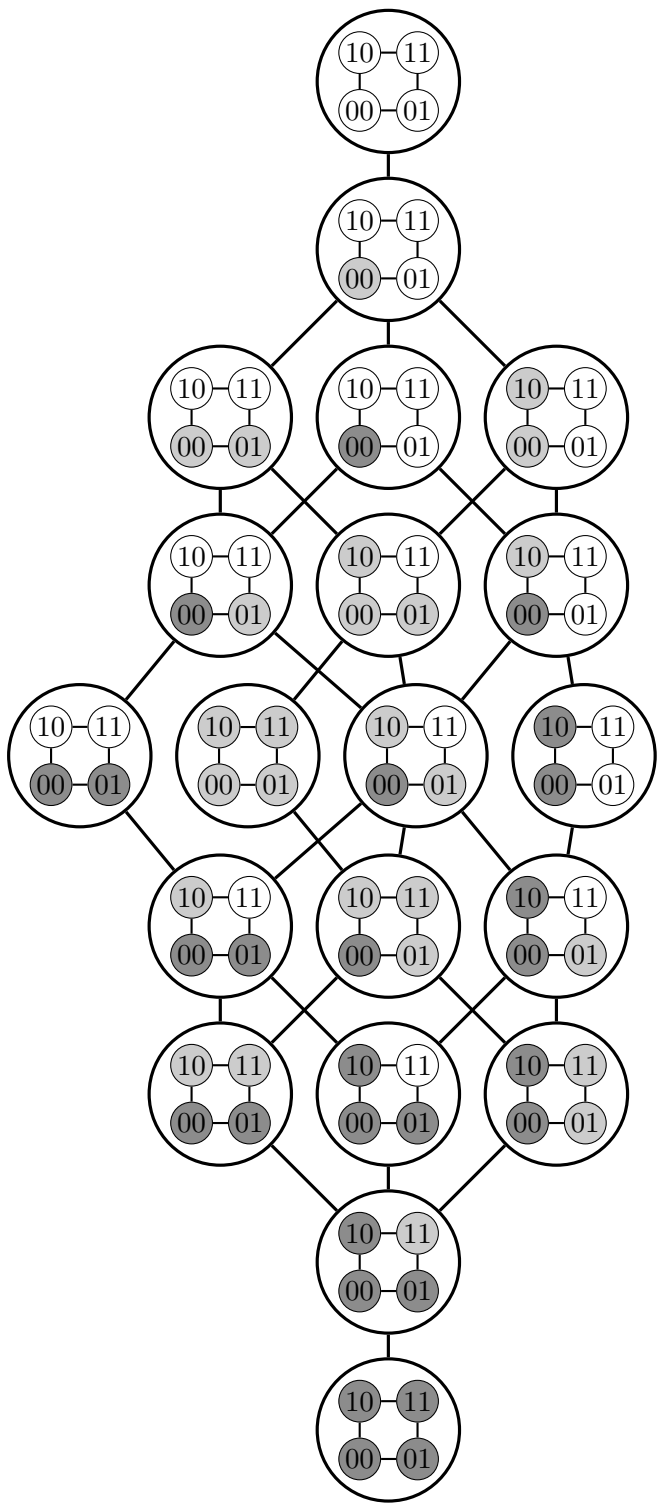

\newpage
\printbibliography

\end{document}

%% file: KM_definitions_7_2018.tex
\theoremstyle{plain}  
\newtheorem{thm}{Theorem}[section]
\newtheorem{lem}[thm]{Lemma}
\newtheorem{prop}[thm]{Proposition}
\newtheorem{cor}[thm]{Corollary}

\theoremstyle{definition}  
\newtheorem{defn}[thm]{Definition}

\theoremstyle{remark}  

\newtheorem{rem}[thm]{Remark}

\newcommand{\bC}{{\bf C}}

\newcommand{\bF}{{\bf F}}

\newcommand{\bRN}{{\bf RN}}

\newcommand{\B}{{\mathbb{B}}}

\newcommand{\N}{{\mathbb{N}}}

\newcommand{\R}{{\mathbb{R}}}

\newcommand{\cR}{{\mathcal R}}

\definecolor{gray85}{gray}{0.85} 
\definecolor{gray8}{gray}{0.8} 
\definecolor{gray7}{gray}{0.7} 
\definecolor{gray6}{gray}{0.6} 
\definecolor{gray5}{gray}{0.5} 
\definecolor{gray4}{gray}{0.4} 
\definecolor{gray35}{gray}{0.35} 

